\documentclass[11pt]{amsart}

\calclayout

\usepackage{hyperref}

\usepackage{color}
\usepackage{verbatim}
\usepackage[noadjust]{cite}
\usepackage{t1enc}
\usepackage[english]{babel}
\usepackage{pgfplots}
\pgfplotsset{compat=1.15}
\usepackage{mathrsfs}
\usetikzlibrary{arrows}

\definecolor{ffqqqq}{rgb}{1.,0.,0.}
\definecolor{uuuuuu}{rgb}{0.26666666666666666,0.26666666666666666,0.26666666666666666}

\usepackage{amsmath,amssymb,amsthm}
\usepackage{mathrsfs,esint}
\usepackage{graphicx,wrapfig}
\usepackage[format=hang]{caption}
\newcommand{\R}{\mathbb{R}}

\newcommand{\N}{\mathbb{N}}

\newcommand{\Ha}{\mathcal{H}}
\newcommand{\W}{\mathcal{W}}

\newcommand{\Leb}{\mathcal{L}}

\newcommand{\eps}{\varepsilon}

\newcommand{\loc}{\text{loc}}

\newcommand{\phii}{\varphi}

\newcommand{\bmat}{\begin{bmatrix}}
\newcommand{\emat}{\end{bmatrix}}

\newcommand{\wtil}{\widetilde}

\newcommand{\st}{\text{ s.t. }}
\newcommand{\oXeps}{\overline X_\eps}

\providecommand*{\vint}[1]{\mathchoice
          {\mathop{\vrule width 5pt height 3 pt depth -2.5pt
                  \kern -9pt \kern 1pt\intop}\nolimits_{\kern -5pt{#1}}}
          {\mathop{\vrule width 5pt height 3 pt depth -2.6pt
                  \kern -6pt \intop}\nolimits_{\kern -3pt{#1}}}
          {\mathop{\vrule width 5pt height 3 pt depth -2.6pt
                  \kern -6pt \intop}\nolimits_{\kern -3pt{#1}}}
          {\mathop{\vrule width 5pt height 3 pt depth -2.6pt
                  \kern -6pt \intop}\nolimits_{\kern -3pt{#1}}}}

\DeclareMathOperator{\Lip}{Lip}

\DeclareMathOperator{\dist}{dist}
\DeclareMathOperator{\codim}{codim}
\DeclareMathOperator{\Mod}{Mod}
\DeclareMathOperator{\diam}{diam}
\DeclareMathOperator{\rad}{rad}

\DeclareMathOperator{\Euc}{Euc}
\DeclareMathOperator{\vcap}{cap}

\numberwithin{equation}{section}
\theoremstyle{plain}
\newtheorem{thm}[equation]{Theorem}

\newtheorem{prop}[equation]{Proposition}

\newtheorem{lem}[equation]{Lemma}

\theoremstyle{definition}

\newtheorem{defn}[equation]{Definition}

\newtheorem{remark}[equation]{Remark}

\newtheorem{example}[equation]{Example}

\makeatletter
\def\blfootnote{\xdef\@thefnmark{}\@footnotetext}
\makeatother

\begin{document}

\title[Self-improvement of  fractional Hardy inequalities]{Self-improvement of fractional Hardy inequalities in metric measure spaces via hyperbolic fillings}

%\blfootnote{2020 {\it Mathematics Subject Classification.}}
%\blfootnote{{\it Keywords and phrases.}}

\author{Sylvester Eriksson-Bique}
\address{Department of Mathematics and Statistics,
P.O. Box 35,
FI-40014 University of Jyväskylä}
\email{sylvester.d.eriksson-bique@jyu.fi}
\author{Josh Kline}
\address{Department of Mathematical Sciences, P.O. Box 210025, University of Cincinnati,
Cincinnati, OH 45221–0025, U.S.A.}
\email{klinejp@ucmail.uc.edu}

\thanks{The first author was supported by Research Council of Finland grants 354241 and 356861. The second author was supported by the NSF Grant DMS-\#2054960, as well as the University
of Cincinnati’s University Research Center summer grant and the Taft Research Center Dissertation
Fellowship Award. The research was conducted during the visit of the second author to University of Jyv\"askyl\"a which was partially funded by the Eemil Aaltonen Foundation through the research group ``Quasiworld network.'' We thank Antti V\"ah\"akangas for suggesting the problem and for many discussions during the early stages of the project.  We would also like to thank Lizaveta Ihnatsyeva for her helpful discussions and insight.}

\subjclass[2020]{35R11
 (26D10; 28A75; 30L15; 31C15; 31E05; 35A23; 46E35)}
\keywords{Fractional Hardy inequality, self-improvement, metric measure space, hyperbolic filling, traces and extensions, Hardy inequality, weighted inequality, Whitney cubes, Besov spaces, Sobolev spaces.}

\date{\today}

\maketitle

\begin{abstract}
In this paper, we prove a self-improvement result for $(\theta,p)$-fractional Hardy inequalities, in both the exponent $1<p<\infty$ and the regularity parameter $0<\theta<1$, for bounded domains in doubling metric measure spaces. The key conceptual tool is a Caffarelli-Silvestre-type argument, which relates fractional Sobolev spaces on $Z$ to Newton-Sobolev spaces in the hyperbolic filling $\oXeps$ of $Z$ via trace results. Using this insight, it is shown that a fractional Hardy inequality in an open subset of $Z$ is equivalent to a classical Hardy inequality in the filling $\oXeps$. The main result is then obtained by applying a new weighted self-improvement result for $p$-Hardy inequalities. The exponent $p$ can be self-improved by a classical Koskela-Zhong argument, but a new theory of regularizable weights is developed to obtain the self-improvement in the regularity parameter $\theta$. This generalizes a result of Lehrb\"ack and Koskela on self-improvement of $d_\Omega^\beta$-weighted $p$-Hardy inequalities by allowing a much broader class of weights. Using the equivalence of fractional Hardy inequalities with Hardy inequalities in the fillings, we also give new examples of domains satisfying fractional Hardy inequalities.
\end{abstract}
\section{Introduction}

\noindent The main theorem of this paper is a self-improvement result for fractional Hardy inequalities. In the process, we develop a Caffarelli-Silvestre-type \cite{CaSi} approach which relates fractional Hardy inequalities to classical Hardy inequalities via trace results. For classical Hardy inequalities, we also give a new weighted self-improvement result. All results are stated for general metric measure spaces, but apply more classically in sub-Riemannian spaces, manifolds and, Euclidean spaces. Even in these classical settings, the results are new.

\subsection{Fractional Hardy inequalities}

 We say that a $p$-Hardy inequality holds for a non-empty open subset $\Omega$ in a metric measure space $(X,d,\mu)$ if for some $C_p\in (0,\infty)$ we have 
\begin{equation}\label{eq:phardy}
\int_\Omega \frac{|u(x)|^p}{d_\Omega(x)^p} d\mu(x) \leq C_p \int_\Omega g_u^p d\mu,
\end{equation}
for all Lipschitz functions $u\in \Lip_c(\Omega)$ that are compactly supported in $\Omega$, where $g_u$ is its minimal $p$-weak upper gradient of $u$, see \cite{HKST, H} for details. Here, $d_\Omega(x)=d(x,X\setminus \Omega)$.  In Euclidean spaces, with $(X,d,\mu)=(\R^n,|\cdot|,\Leb^n)$, the $p$-Hardy inequality can be equivalently stated for all continuously differentiable functions $u\in C^1_0(\Omega)$ with compact support in $\Omega$ and  with $|\nabla u|$ replacing $g_u$. For more detailed definitions, see Section \ref{sec:preliminaries}. Both the Euclidean inequality, and the general metric measure variant, have been extensively studied, see \cite{Hardy,KZ, L, L2, Brezis,C, HS, Kline, Wannebo} and references therein. Good surveys on the topic are \cite{KufnerPersson, OK, Davies}. 

It is natural to also consider the fractional version of the previous inequality, where the Dirichlet-type energy on the right is replaced by a fractional energy. The $(\theta,p)$-fractional Hardy inequality is:
\begin{equation}\label{eq:frachardy}
\int_\Omega \frac{|u(x)|^p}{d_\Omega(x)^{\theta p}} d\mu(x) \leq C_{\theta,p} \int_X\int_X \frac{|u(x)-u(y)|^p}{d(x,y)^{\theta p} \mu(B(x,d(x,y)))} d\mu(y) d\mu(x).
\end{equation}
The energy on the right is one of the definitions of a fractional Sobolev energy used in the definition of Besov spaces, see e.g. \cite{Adams,CaSi, GKS,BBS}.
The fractional Hardy inequality has previously been studied in e.g. \cite{dyda,KufnerPersson,LMV,IMV,DydaLV}. 

Hardy inequalities enjoy self-improvement phenomena: if a given open set $\Omega$ satisfies a $p$-Hardy inequality, then it also satisfies a $q$-Hardy inequality with $q\in (p-\epsilon,p+\epsilon)$ for some $\epsilon>0$.  This was first shown by Koskela and Zhong \cite{KZ}, and was further studied and reproved in \cite{EB1, EB2, L, Korte}. Such self-improvement results were known earlier for Muckenhoupt weights \cite{Stein} and for Poincar\'e inequalities \cite{KeZ}. There are also weighted self-improvement results, see \cite{L,KufnerPersson}. Such improvement results are important in applications to partial differential equations and regularity of quasiconformal maps, see \cite{KZ} for further discussion and references.

Our main theorem gives a self-improvement result for fractional Hardy inequalities in doubling metric measure spaces. 

\begin{thm}\label{thm:FracHardyImprovement}
    Let $(Z,d,\nu)$ be a complete, doubling metric measure space.  Let $E\subset Z$ be a closed set such that $Z\setminus E$ is bounded and satisfies a $(\theta_0,p_0)$-Hardy inequality, with constant $C_{\theta_0,p_0}$, for some $0<\theta_0<1$ and $1<p_0<\infty$.  Then there exists $\eps_0>0$, depending only on $\theta_0$, $p_0$, $C_{\theta_0,p_0}$, and $C_\nu$, such that $Z\setminus E$ satisfies a $(\theta,p)$-Hardy inequality for all $\theta\in(\theta_0-\eps_0,\theta_0+\eps_0)$ and $p\in (p_0-\eps_0,p_0+\eps_0)$.    
\end{thm}
\noindent Here, $C_\nu$ is the doubling constant of $\nu$, see Section~\ref{sec:preliminaries}.

Self-improvement results of the above form were recently obtained in \cite{IMV} for a \emph{pointwise} version of $(\theta,p)$-Hardy inequalities under the additional assumptions that $Z$ is geodesic and reverse doubling, that $E\subset Z$ satisfies a fractional $(\theta,p)$-capacity density condition, see Definition~\ref{def:FracCapDensity}, and under some restrictions on the parameters $\theta$ and $p$.  In fact, the results there were obtained for pointwise $(\theta,p,q)$-Hardy inequalities, where the right hand side of \eqref{eq:frachardy} is replaced with a $F^\theta_{p,q}$-Triebel-Lizorkin energy.  In their argument, the authors show that the fractional capacity density condition is equivalent to a pointwise fractional Hardy inequality. In conjunction with results from \cite{LMV}, they then obtain a deep self-improvement result for the fractional capacity density condition. The pointwise fractional Hardy inequality implies the (integral) fractional Hardy inequality under some restrictions of the parameters, and thus \cite{LMV} implies Theorem \ref{thm:FracHardyImprovement} under the additional assumption of a pointwise Hardy inequality. While our results are limited to the case that $p=q$, our different approach allows us to remove the pointwise fractional Hardy inequality assumption, and so our results apply to domains such as the Euclidean punctured unit ball, see Examples~\ref{ex:punctured ball} and \ref{ex:Zero Capacity}, for which this condition fails.  Moreover, as we utilize the trace and extension relationship between Sobolev and Besov functions via the hyperbolic filling, described below, we do not need to assume that $Z$ is geodesic or reverse doubling, and so we are able to obtain our results for a full range of parameters $0<\theta<1$ and $1<p<\infty$.  For more on pointwise Hardy inequalities and capacity density conditions, see for example \cite{KM,Korte,EB1,EB2,Haj,LMV}.  See also the classical work of Lewis on the self-improvement of the $p$-capacity density condition \cite{Lewis}.  We now describe the crucial tools that lead to a proof of Theorem~\ref{thm:FracHardyImprovement}.

\subsection{Equivalence with $p$-Hardy inequalities in the hyperbolic filling}

As mentioned above, there are a number of known results on the self-improvement for the $p$-Hardy inequality, \cite{KZ,L2,L}. Motivated by this, one might seek to prove self-improvement for the fractional Hardy inequality by either mimicking the proofs of these results, or by reducing the fractional Hardy inequality results to those for Hardy inequalities. Here, we take the latter approach, which is motivated by the Caffarelli-Silvestre theory \cite{CaSi} (see also \cite{Baudoin, EBGKSS} for more general versions) and the trace theory of Sobolev spaces \cite{BBS} using hyperbolic fillings. 

The technique of hyperbolic fillings was introduced in \cite{BK,BP, BSaks} and has origins in a construction from Gromov \cite[P.99 (b)]{Gromov}. Since this technique has mainly been developed for compact spaces, we will now focus on the case where $Z$ is additionally assumed to be compact. This assumption is removed later through a localization argument. In this approach, given $(Z,d)$ one constructs another metric space, a uniformized hyperbolic filling $(\oXeps,d_\eps)$, whose boundary is bi-Lipschitz equivalent with $(Z,d)$. In \cite{BBS}, the authors considered $(Z,d,\nu)$ to be a compact, doubling metric measure space, and equipped the uniformized hyperbolic filling of $Z$ with a family of measures $\mu_\beta$ induced from $\nu$.  The main result in \cite{BBS} shows that Besov spaces on $Z$ - functions for which the right hand side of \eqref{eq:frachardy} is finite - arise as traces of Newton-Sobolev functions, in the sense of \cite{shanmun,cheeger,HKST}, on $\oXeps$. The Dirichlet energy of Newton-Sobolev functions is the right hand side of the $p$-Hardy inequality \eqref{eq:phardy}. Thus, the trace theory forms a bridge between the Dirichlet energies in \eqref{eq:phardy} with the fractional energies in \eqref{eq:frachardy}. Utilizing this, we prove that the fractional $(\theta,p)$-Hardy inequality in $(Z,d,\nu)$ is equivalent to a $p$-Hardy inequality in $(\oXeps,d_\eps,\mu_\beta)$.  While this is simple to state, the proof involves several nontrivial estimates.

\begin{thm}\label{thm:HypFill Hardy}
    Let $(Z,d,\nu)$ be a compact, doubling metric measure space, with $\diam(Z)<1$.  Let $E\subset Z$ be a closed set, and let $1<p<\infty$, $0<\theta<1$, and $\beta>0$ be such that $\beta/\eps=p(1-\theta)$.  Then, $Z\setminus E$ satisfies a $(\theta,p)$-Hardy inequality if and only if $\oXeps\setminus E$ satisfies a $p$-Hardy inequality with respect to $\mu_\beta$.
\end{thm}
The interlinking of the parameters $\beta$, $\theta$, and $p$ given by 
\begin{equation}\label{eq:relationship}
\frac{\beta}{\eps}=p(1-\theta)
\end{equation}
relates to the trace theory in \cite{BBS}, see Theorem~\ref{thm:HypFill} below.  Here $\eps>0$ is a fixed parameter from the hyperbolic filling construction.

 Theorem \ref{thm:HypFill Hardy} is conceptually simple, and it reinforces the general idea that non-local problems can be studied through proving equivalences with local problems in hyperbolic fillings. By the self-improvement results for $p$-Hardy inequalities established in  \cite{KZ}, we see that a $p$-Hardy inequality in the space $(\oXeps,d_\eps,\mu_\beta)$ implies a $q$-Hardy inequality for $q\in (p-\epsilon,p+\epsilon)$. Using the equivalence in Theorem \ref{thm:HypFill Hardy}, it then follows that a $(\theta,p)$-Hardy inequality in $Z$ gives us a $(q,\theta')$-Hardy inequality for some \emph{other} value  $\theta'$. The reason for this is that altering $p$ to $q$ in \eqref{eq:relationship} results in $\theta$ changing, as the results of \cite{KZ} keep $\beta$ fixed. This gives self-improvement along a single curve, see Figure \ref{fig:KoskelaZhongImprovement} in Section~\ref{sec:Self-Improvement}.  

Theorem~\ref{thm:HypFill Hardy} also provides a method to readily obtain results pertaining to fractional Hardy inequalities directly from known corresponding results for $p$-Hardy inequalities.  We illustrate this in Section~\ref{sec:Examples} with Proposition~\ref{prop:JuhaFractional}, which gives a sufficient condition for $(\theta,p)$-Hardy inequalities in terms of Assouad codimensions.  Such a result was established for $p$-Hardy inequalities in \cite{L2}, and Theorem~\ref{thm:HypFill Hardy} allows us to easily obtain the fractional analog.  Using this proposition, we show that the punctured unit ball in $\R^n$ satisfies a $(\theta,p)$-Hardy inequality when $\theta p<n$, see Example~\ref{ex:punctured ball}.    

\subsection{Towards full improvement: regularizable weights}

In order to use the hyperbolic filling to get full self-improvement of the fractional Hardy inequality, that is, with respect to the parameters $\theta$ and $p$ \emph{independently}, \eqref{eq:relationship} requires that we are able to change the $\beta$-parameter. Changing the $\beta$-parameter in $(\oXeps,d_\eps,\mu_\beta)$ corresponds to adjusting the reference measure $\mu_\beta$ with a doubling weight $w$. Thus, one needs a \emph{weighted} self-improvement result for $p$-Hardy inequalities. The only general result of this type that we are aware of is \cite{L2}, where certain distance weights are used. However, the weight $w$ is not of this type. This prompts us to show a weighted self-improvement for a more general class of weights, which includes the weights needed to alter $\mu_\beta$ to $\mu_{\beta'}$. 

The weighted self-improvement result may be of general interest. The general question is the following. Under which assumptions on $w\in L^1_{\rm loc}(X,\mu)$ (and the space $X$) does \eqref{eq:phardy} imply
\begin{equation}\label{eq:phardyweight}
\int_\Omega \frac{|u(x)|^p}{d_\Omega(x)^p} wd\mu \leq C_p \int_\Omega g_u^p wd\mu?
\end{equation}
We give a sufficient condition with two assumptions. The assumptions are a bit technical, but can be summarized as follows.
\begin{enumerate}
    \item $p$-admissibility: The measure $wd\mu$ is doubling and supports a $(1,p)$-Poincar\'e inequality.
    \item $\delta$-regularizability (in $\Omega$): Whenever $B_x=B(x,d_\Omega(x)/8)$ and $B_y = B(y,d_\Omega(y)/8)$ are two close enough balls, then \[
    |w_{B_x}-w_{B_y}|\leq \delta w_{B_x}.
    \]
    Here $w_A=\frac{1}{\mu(A)}\int_A w d\mu$ for any Borel set $A$ with $\mu(A)>0$.
\end{enumerate}
The latter of these is a natural condition involving \emph{Whitney scales}, i.e. balls $B(x,d_\Omega(x)/C)$ for $C>1$. The details are given in Definition~\ref{def:p-admissibile Whitney} and Definition \ref{def:reg weights}. In terms of these conditions, we obtain the following weighted self-improvement result; see also Remark \ref{rem:Juhas weights}.

\begin{thm}\label{thm:delta reg improvement}
    Let $(X,d,\mu)$ be a complete, doubling metric measure space, and let $1<p<\infty$. Suppose that an open set $\Omega\subset X$ satisfies a $p$-Hardy inequality with respect to $\mu$, with constant $C_p\ge 1$.  Then there exists a constant $\delta_p>0$, such that if $w$ is 
    \begin{itemize}
        \item[(i)]$p$-admissible, 
        \item[(ii)]$\delta$-regularizable in $\Omega$ for some $0<\delta\le\delta_p$,
    \end{itemize}
    then $\Omega$ satisfies a $p$-Hardy inequality with respect to $w\,d\mu$, with constant depending only on $p$, $C_\mu$, and the $p$-admissibility constants of $w$.  
    
    We may take $\delta_p=\frac{p}{2C_0}(2C_p)^{-1/p}$, where $C_0$ is the constant from Lemma~\ref{lem:regularized gradient} which depends only on $C_\mu$.
\end{thm}
While the above theorem is stated for $p$-admissible weights, this assumption can be weakened to require that the weighted measure supports a Poincar\'e inequality on Whitney balls in $\Omega$, see Definition~\ref{def:p-admissibile Whitney}.  Theorem~\ref{thm:delta reg improvement}, with this slightly weaker admissibility condition, then recovers the weighted self-improvement result \cite[Theorem~1]{L2}, given for weights of the form $d_\Omega^\beta$, see Remark~\ref{rem:Juhas weights}.  

We use this theorem to obtain the full self-improvement in Theorem \ref{thm:FracHardyImprovement}. However, this result is also new for the classical $p$-Hardy inequality, and it strengthens the main results of \cite{L} by allowing more general weights. In the one dimensional Euclidean case, weighted Hardy inequalities and characterizations thereof have appeared in \cite{Muckenhoupt}.

The condition of $\delta$-regularizability arose from examining the proofs in \cite{dyda,L,L2}. The rough idea is that the Hardy inequality contains two levels of estimates: estimates at Whitney scales, and estimates at large scales. The Whitney scale information relates to $p$-admissibility. The self-improvement of the large scale estimates is related to $\delta$-regularizability -- morally since a $\delta$-regularizable weight will not alter the large scale estimate too drastically. A close examination of the proofs in \cite{dyda,L,L2} shows that the same phenomenon has been used previously. The appearance of Whitney scales is also natural from the perspective of characterizations for Hardy inequalities that have appeared in the literature such as \cite{I}.

Using Theorem~\ref{thm:delta reg improvement}, it follows that a $p$-Hardy inequality in the hyperbolic filling $(\oXeps,d_\eps,\mu_\beta)$ persists when changing the measure from $\mu_\beta$ to $\mu_{\beta'}$.  Theorem~\ref{thm:HypFill Hardy} then gives us that a $(\theta,p)$-Hardy inequality on $Z$ implies a $(\theta',p)$-Hardy inequality, with $\theta'$ arising from the perturbation of $\beta'$ in \eqref{eq:relationship}.  Combining this argument with the self-improvement results from \cite{KZ}, we are able to pass from the one-parameter self-improvement, see Figure~\ref{fig:KoskelaZhongImprovement} in Section~\ref{sec:Self-Improvement}, to the full self-improvement result, see Figure~\ref{fig:full improvement}, which gives us Theorem~\ref{thm:FracHardyImprovement}.

\subsection{Organization of the paper}
We will first present the necessary background in Section \ref{sec:preliminaries}. Then, in Section \ref{sec:Reg Weights} we give a proof of the weighted self-improvement of $p$-Hardy inequalities from Theorem \ref{thm:delta reg improvement}. This section is independent and a reader interested only in weighted self-improvement results can read this independent of the other parts of the paper. Hyperbolic fillings are defined in Section \ref{sec:HypFill}, and Theorem \ref{thm:HypFill Hardy} on equivalences is proved in Section \ref{sec:FracHardy-HypFill}. The pieces are put together in Section \ref{sec:Self-Improvement}, where we first present a localization argument, allowing us to reduce the proof of Theorem~\ref{thm:FracHardyImprovement} to the case of a compact metric measure space.  We then use weighted self-improvement and \cite{KZ} to improve the $p$-Hardy inequality in the hyperbolic filling, and by passing through the equivalence, we obtain the self-improvement for fractional $(\theta,p)$-Hardy inequalities in $Z$. Finally, in Section \ref{sec:Examples}, we present some applications of  Theorem~\ref{thm:HypFill Hardy}, by proving Proposition~\ref{prop:JuhaFractional} and Example~\ref{ex:punctured ball}.

\section{Preliminaries}\label{sec:preliminaries}
In this section we provide the necessary background information and definitions used throughout the paper.  In Section~\ref{sec:HypFill}, we will introduce particular notation to denote the uniformized hyperbolic filling $(\oXeps,d_\eps,\mu_\beta)$ of a compact metric measure space $(Z,d,\nu)$.  To avoid confusion with this particular notation, however, we will present the general background notions in this section using $(X,d,\mu)$ to denote a generic metric measure space.  

Throughout this paper, we let $C>0$ denote a constant which, unless otherwise specified, depends only on the structural constants of the metric measure space, such as the doubling constant for example.  Its precise value is not of interest to us, and may change with each occurrence, even within the same line.  Furthermore, given quantities $A$ and $B$, we will often use the notation $A\simeq B$ to mean that there exists a constant $C\ge 1$ such that $C^{-1} A\le B\le CA$. Likewise, we use $A\lesssim B$ and $A\gtrsim B$ if the left and right inequalities hold, respectively.

\subsection{Doubling measures}
For $x\in X$ and $r>0$, we denote by $B(x,r)$ the ball $\{y\in X:d(x,y)<r\}$.  Given a ball $B:=B(x,r)\subset X$, and a constant $C>0$, we denote by $CB$ the ball $B(x,Cr)$.  Given $x\in X$ and $E,F\subset X$, we also denote by $d(E,F)$ and $d(x,E)$, the distance between sets $E$ and $F$, and the distance between $x$ and $E$, respectively, as induced by the metric $d$.  

We say that a Borel measure $\mu$ is \emph{doubling}, if there exists a constant $C_\mu\ge 1$ such that 
\begin{equation*}
    0<\mu(B(x,2r))\le C_\mu\mu(B(x,r))<\infty
\end{equation*}
for all $x\in X$ and $r>0$. By iterating this condition, there exist constants $Q>0$ and $C\ge 1$, depending only on $C_\mu$, such that for all $x\in X$, $0<r\le R$, and $y\in B(x,R)$, we have
\begin{equation*}\label{eq:LMB}
    \frac{\mu(B(y,r))}{\mu(B(x,R))}\ge C^{-1}\left(\frac{r}{R}\right)^Q.
\end{equation*}
Here, we may take $Q=\log_2C_\mu$.  If $\mu$ is doubling, many classical results from harmonic analysis hold in $(X,d,\mu)$, such as the Lebesgue differentiation theorem and Hardy-Littlewood maximal function theorem, see \cite{H} for example.  Such metric measure spaces also admit the following Whitney decomposition of an open set.  For proof of the following lemma see \cite[Section~4]{HKST}, for example.

\begin{lem}\label{lem:Whitney Decomp}
Let $(X,d,\mu)$ be a doubling metric measure space, and let $\Omega\subset X$ be open, with $X\setminus\Omega\ne\varnothing$. Then there exists a countable collection of balls $\mathcal{W}_\Omega:=\{B(x_i,r_i)=:B_i\}_{i\in\N}$ such that 
\[
\Omega=\bigcup_i B_i
\]
and that 
\[
\sum_i\chi_{6B_i}\le C,
\]
with \[r_i=\frac{1}{8}d_\Omega(x_i).\]
Here the constant $C\ge 1$ depends only on the doubling constant of $\mu$. Moreover, there exists a Lipschitz partition of unity $\{\phii_i\}_{i}$ subordinate to $\mathcal{W}_\Omega$.  That is, each $\phii_i$ is $C/r_i$-Lipschitz, with $0\le\phii_i\le\chi_{2B_i}$, and 
\[
\sum_{i}\phii_i(x)=1
\]
for each $x\in\Omega$.  Furthermore, there exists $C\ge 1$, depending only on the doubling constant of $\mu$, such that $\phii_i|_{B_i}\ge C^{-1}$ for each $i$.    
\end{lem}

\subsection{Newton-Sobolev and Besov classes}
Let $1\le p<\infty$.  Given a family $\Gamma$ of non-constant, compact, rectifiable curves, we define the \emph{$p$-modulus} of $\Gamma$ by 
\[
\Mod_p(\Gamma):=\inf_\rho\int_ X\rho^p\,d\mu,
\]
where the infimum is over all non-negative Borel functions $\rho:X\to[0,\infty]$ such that $\int_\gamma\rho\,ds\ge 1$ for all $\gamma\in\Gamma$.  Given a function $u:X\to\overline\R$, we say that a Borel function $g:X\to[0,\infty]$ is an \emph{upper gradient} of $u$ if the following holds for all non-constant, compact, rectifiable curves $\gamma:[a,b]\to X$:
\begin{equation*}
    |u(\gamma(b))-u(\gamma(a))|\le\int_\gamma g\,ds
\end{equation*}
whenever $u(\gamma(b))$ and $u(\gamma(a))$ are both finite, and $\int_\gamma g\,ds=\infty$ otherwise. We say that $g$ is a \emph{$p$-weak upper gradient of $u$} if the family of curves where the above inequality fails has $p$-modulus zero. 

We define $\wtil N^{1,p}(X,\mu)$ to be the class of all functions in $L^p(X,\mu)$ which have an upper gradient belonging to $L^p(X,\mu)$.  We then define 
\begin{equation*}
    \|u\|_{\wtil N^{1,p}(X,\mu)}:=\|u\|_{L^p(X,\mu)}+\inf_g\|g\|_{L^p(X,\mu)},
\end{equation*}
where the infimum is taken over all upper gradients $g$ of $u$.  Defining the equivalence relation $\sim$ in $\wtil N^{1,p}(X,\mu)$ by $u\sim v$ if and only if $\|u-v\|_{\wtil N^{1,p}(X,\mu)}=0$, we then define the \emph{Newton-Sobolev space} $N^{1,p}(X,\mu)$ to be $\wtil N^{1,p}(X,\mu)/\sim$, equipped with the norm $\|\cdot\|_{N^{1,p}(X,\mu)}:=\|\cdot\|_{\wtil N^{1,p}(X,\mu)}$.  We can similarly define $N^{1,p}(\Omega,\mu)$ for any open $\Omega\subset X$.  Each $u\in N^{1,p}(X,\mu)$ has a \emph{minimal $p$-weak upper gradient}, denoted $g_u$, which is unique $\mu$-a.e., see \cite[Chapter~6]{HKST}. For more background and historical references on modulus, upper gradients and Sobolev spaces see \cite{HKST,BBS,H,shanmun,cheeger}.

For $0<\theta<1$ and $1\le p<\infty$, we define the \emph{Besov energy} of a function $u\in L^1_\loc(X,\mu)$ by 
\begin{equation}\label{eq:Besov energy}
    \|u\|^p_{B^\theta_{p,p}(X,\mu)}:=\int_X\int_X\frac{|u(x)-u(y)|^p}{d(x,y)^{\theta p}\mu(B(x,d(x,y)))}d\mu(y)d\mu(x).
\end{equation}
The \emph{Besov space} $B^\theta_{p,p}(X,\mu)$ is then defined as the set of all functions in $L^p(X,\mu)$ for which this energy is finite.  If $X=\R^n$ and $\mu=\Leb^n$ is the standard $n$-dimensional Lebesgue measure, then $B^\theta_{p,p}(X,\mu)$ corresponds to the fractional Sobolev space $W^{\theta,p}(\R^n)$.  It was shown in \cite{GKS} that when $(X,d,\mu)$ is a doubling metric measure space supporting a $(1,p)$-Poincar\'e inequality, defined below, $B^\theta_{p,p}(X,\mu)$ coincides with the real interpolation space $(L^p(X,\mu),KS^{1,p}(X,\mu))_{\theta,p}$, where $KS^{1,p}(X,\mu)$ is the Korevaar-Schoen space. See the nice survey \cite{DNPVH} and \cite{Adams, BBS,EBGKSS,CaSi,GKS,JW} for more discussion and references.

\subsection{Poincar\'e inequalities and $p$-admissible weights}
Let $1\le p,q<\infty$.  We say that $(X,d,\mu)$ supports a \emph{$(q,p)$-Poincar\'e inequality} if there exist constants $C_{PI},\lambda\ge 1$ such that the following holds for all balls $B\subset X$ and function-upper gradient pairs $(u,g)$:
\begin{equation*}
    \left(\fint_B|u-u_B|^q\,d\mu\right)^{1/q}\le C_{PI}\rad(B)\left(\fint_{\lambda B}g^p\,d\mu\right)^{1/p}.
\end{equation*}
Here and throughout the paper, we use the notation 
\[
u_B:=\fint_B u\,d\mu=\frac{1}{\mu(B)}\int_Bu\,d\mu.
\]

If $(X,d,\mu)$ is a geodesic space supporting a $(q,p)$-Poincar\'e inequality, then we may take $\lambda=1$, as shown in \cite{HaKo}.  We note that if $(X,d,\mu)$ is doubling and supports a $(1,p)$-Poincar\'e inequality, then Lipschitz functions are dense in $N^{1,p}(X,\mu)$, see \cite[Theorem~8.2.1]{HKST}.

By a \emph{weight}, we mean a non-negative function $w\in L^1_\loc(X,\mu)$.  To such a weight, we associate the weighted measure $\mu_w:= w\,d\mu$. We say $w$ is a \emph{doubling weight} if the measure $\mu_w$ is doubling. Moreover, we say that a weight $w$ is \emph{$p$-admissible} if $\mu_w$ is doubling and supports a $(1,p)$-Poincar\'e inequality. The terminology is due to the classic reference of Heinonen, Kilpeläinen and Martio \cite{HKM}.  While Theorem~\ref{thm:delta reg improvement} is stated in terms of $p$-admissible weights, the result holds true under a slightly weaker $p$-admissibility condition, namely that $\mu_w$ supports a Poincar\'e inequality on balls at Whitney scales, see Remark~\ref{rem:Juhas weights}.  In terms of Lemma~\ref{lem:Whitney Decomp} above, we give the following definition:

\begin{defn}\label{def:p-admissibile Whitney}
    Let $\Omega\subset X$ be open, and let $1<p<\infty$.  We say that a non-negative function $w\in L^1_\loc(X,\mu)$ is \emph{$p$-admissible at Whitney scales in $\Omega$} if the following conditions are satisfied:
    \begin{itemize}
        \item[(i)] $\mu_w:=w\,d\mu$ is doubling,  
        \item[(ii)] there exist $C,\lambda\ge 1$ such that for all $x\in\Omega$, $0<r\le 3d_\Omega(x)/4$, and function-upper gradient pairs $(u,g)$, we have
        \[
        \fint_{B(x,r)}|u-u_{B(x,r)}|d\mu_w\le Cr\left(\fint_{\lambda B(x,r)}g^p\,d\mu_w\right)^{1/p},
        \]
        where here $u_{B(x,r)}:=\frac{1}{\mu_w(B(x,r))}\int_{B(x,r)} u\,d\mu_w$.
     \end{itemize} 
       
\end{defn}

Any Muckenhoupt $A_p$-weight is a $p$-admissible weight. In Euclidean spaces, this was shown in \cite[Section 15.22]{HKM}. For metric measure spaces that satisfy a $(1,1)$-Poincar\'e inequality this is an unstated folklore result, and can be obtained by applying the proof from \cite[Section 15.22]{HKM} together with a Riesz-potential characterization of the Poincar\'e inequality from e.g. \cite{Keith}. Alternatively, a more direct argument is to modify the proof of \cite[Theorem 4]{B} which uses the equivalent characterization of the Poincar\'e inequality using a pointwise condition, see \cite[Theorems 3.1 and 3.2]{HaKo}. Thus, Muckenhoupt weights give a large class of examples of admissible weights, see e.g. \cite{BBK,B,KKM}. Muckenhoupt weights on the other hand are often constructed as distance weights, see \cite{DILTV,Aimar}. In our case, we will also construct $p$-admissible weights using distance functions, but the Poincar\'e inequality in our case can be checked more directly than through the Muckenhoupt condition. 

\subsection{Hardy inequalities}
\begin{defn}\label{def:p-Hardy}
    Let $\Omega\subset X$ be open, and let $1<p<\infty$. We say that $\Omega$ satisfies a \emph{$p$-Hardy inequality} (with respect to $\mu$) if there exist a constant $C_p\ge 1$ such that the following holds for all $u\in\Lip_c(\Omega)$:
    \begin{equation*}
        \int_\Omega\frac{|u(x)|^p}{d(x,X\setminus\Omega)^p}d\mu(x)\le C_p\int_{\Omega}g_u^p\,d\mu.
    \end{equation*}
\end{defn}

As shown by Koskela and Zhong in \cite{KZ}, $p$-Hardy inequalities possess the following self-improvement property:  

\begin{thm}\cite[Theorem~1.2]{KZ}\label{thm:KoskelaZhong}
    Let $(X,d,\mu)$ be a doubling metric measure space supporting a $(1,p)$-Poincar\'e inequality. Let $\Omega\subset X$ be an open set, and suppose that $\Omega$ satisfies a $p$-Hardy inequality with respect to $\mu$.  Then there exists $\eps_1:=\eps_1(p,C_p,C_\mu,\lambda,C_{PI})$, such that $\Omega$ satisfies a $q$-Hardy inequality with respect to $\mu$ for all $q\in(p-\eps_1,p+\eps_1)$.  Moreover, there exists $C_1:=C_1(p,C_p,C_\mu,\lambda,C_{PI})$ such that $C_q=C_1$ for all such $q$.
\end{thm}

In addition to the above self-improvement result, $p$-Hardy inequalities also self-improve with respect to weights of the form $d(\cdot, X\setminus\Omega)^\beta$, as shown by Lehrb\"ack in \cite[Theorem~1]{L} and \cite[Proposition~6.4]{L2}.

\begin{thm}\cite[Theorem~1]{L}\label{thm:Lehrback}
Let $(X,d,\mu)$ be a doubling metric measure space supporting a $(1,p)$-Poincar\'e inequality. Let $\Omega\subset X$ be an open set, let $\beta_0\in\R$, and suppose that $\Omega$ satisfies a $p$-Hardy inequality with respect to $d(\cdot,X\setminus\Omega)^{\beta_0}\,d\mu$.  Then there exists $\eps:=\eps(p,\beta_0,C_p,C_\nu,\lambda, C_{PI})$ such that $\Omega$ satisfies a $p$-Hardy inequality with respect to $d(\cdot,X\setminus\Omega)^\beta\,d\mu$ for all $\beta\in(\beta_0-\eps,\beta_0+\eps)$, with constant independent of $\beta$.    
\end{thm}

\noindent In Section~\ref{sec:Reg Weights}, we generalize this result to apply to a larger class of weights, see Theorem~\ref{thm:delta reg improvement}.

The main object of our study is the following fractional Hardy inequality:

\begin{defn}\label{def:Frac Hardy}
    Let $\Omega\subset X$ be open, let $1<p<\infty$, and let $0<\theta<1$.  We say that $\Omega$ satisfies a \emph{$(\theta,p)$-Hardy inequality} if there exists $C_{\theta,p}\ge 1$ such that the following holds for for all $u\in\Lip_c(\Omega)$:
    \begin{equation*}
        \int_{\Omega}\frac{|u(x)|^p}{d(x,X\setminus\Omega)^{\theta p}}d\mu(x)\le C_{\theta,p}\int_X\int_X\frac{|u(x)-u(y)|^p}{d(x,y)^{\theta p}\mu(B(x,d(x,y)))}d\mu(y)d\mu(x).
    \end{equation*}
\end{defn}
We note that the right-hand side of the $(\theta,p)$-Hardy inequality is given by the Besov energy of $u$.  As this energy is nonlocal, it receives some contribution from the region where $u$ vanishes. That is, the regions $X\setminus\Omega$ contribute to the energy of the right-hand side.  This is in contrast to the $p$-Hardy inequality, for which the energy on the right-hand side is local: since $u$ vanishes on $X\setminus\Omega$, it follows that $g_u=0$ there.  For sufficient conditions allowing for integration over $\Omega$ in the right-hand side of above inequality, see \cite{ILTV}.

\section{Regularizable weights and weighted $p$-Hardy inequalities}\label{sec:Reg Weights}

In this section, we consider a metric measure space $(X,d,\mu)$, with $\mu$ a doubling measure, and an open set $\Omega\subset X$.  Throughout this section, we use the notation 
\begin{equation*}
    d_\Omega(\cdot):=d(\cdot,X\setminus\Omega).
\end{equation*}
We show that if $\Omega$ satisfies a $p$-Hardy inequality with respect to $\mu$, then it also satisfies a $p$-Hardy inequality with respect to the weighted measure $w\,d\mu$ whenever $w$ is a $p$-admissible weight which is \emph{$\delta$-regularizable}, defined below, for sufficiently small $\delta>0$.  Our result, Theorem~\ref{thm:delta reg improvement}, generalizes the self-improvement result Theorem~\ref{thm:Lehrback} obtained in \cite{L,L2} by Lehrb\"ack for weights of the form $d_\Omega^\beta$, see Remark~\ref{rem:Juhas weights}.  We will use Theorem~\ref{thm:delta reg improvement} in Section~\ref{sec:Self-Improvement} to prove self-improvement of the fractional Hardy inequality.

Recall that we define a weight to be a non-negative function $w\in L^1_{\rm loc}(X,\mu)$. We associate to such $w$ a weighted measure $d\mu_w=wd\mu$, and call $w$ doubling if $\mu_w$ is doubling. To simplify notation, we use $w$ in place of $\mu_w$ and write $w(A)=\mu_w(A)$ for the weighted measure of a measurable set $A$.  We now define $\delta$-regularizable weights, roughly speaking, as those whose average values do not change much at Whitney scales:

\begin{defn}\label{def:reg weights}
    Let $\delta>0$.  We say that a doubling weight is \emph{$\delta$-regularizable} (in $\Omega$) if whenever $x,x'\in\Omega$ and $B_x:=B(x,d_\Omega(x)/8)$ and $B_{x'}:=B(x',d_\Omega(x')/8)$ are such that $2B_{x'}\cap B_x\ne\varnothing$, we have
    \begin{equation}\label{eq:delta reg condition}
        |w_{B_x}-w_{B_{x'}}|\le\delta w_{B_x}.
    \end{equation}
    Here we again use the notation
    \[
    w_B:=\fint_B w\,d\mu=\frac{w(B)}{\mu(B)}.
    \]
\end{defn}

Let $\W_\Omega=\{B_i\}_i$ and $\{\phii_i\}_{i}$ be the Whitney decomposition of $\Omega\subset X$ and associated Lipschitz partition of unity, respectively, given by Lemma~\ref{lem:Whitney Decomp}. We define the discrete convolution $\wtil w$ of a weight $w$ by 
\begin{equation}\label{eq:weight convolution}
\wtil w(x)=\sum_iw_{B_i}\phii_i(x).
\end{equation}
The key property of $\delta$-regularizable weights is the following estimate for the upper gradient of the discrete convolution of a $\delta$-regularizable weight: 

\begin{lem}\label{lem:regularized gradient}
Let $\Omega\subset X$ be an open set. 
    There exists a constant $C_0\ge 1$, depending only on $C_\mu$, such that if $w$ is a $\delta$-regularizable weight in $\Omega$, then 
    \[
    g_{\wtil w}\le C_0\delta d_\Omega^{-1}\wtil w
    \]
    a.e. in $\Omega$.
\end{lem}

\begin{proof}
    Suppose that $x\in B_i$.  Then by the partition of unity, it follows that 
    \[
    \wtil w(x)=w_{B_i}+\sum_{j}\left(w_{B_j}-w_{B_i}\right)\phii_j(x)=w_{B_i}+\sum_{j:\,2B_j\cap B_i\ne\varnothing}\left(w_{B_j}-w_{B_i}\right)\phii_j(x).
    \]
    Since $w$ is $\delta$-regularizable, and by bounded overlap of $\W_\Omega$ and the properties of the Lipschitz partition of unity, we then have that 
    \[
    g_{\wtil w}(x)\lesssim\sum_{j:\,2B_j\cap B_i\ne\varnothing}\left|w_{B_j}-w_{B_i}\right|r_j^{-1}\lesssim \delta w_{B_i} d_\Omega(x)^{-1}\lesssim\delta d_\Omega(x)^{-1}\wtil w(x),
    \]
    with constant depending only on the doubling constant of $\mu$.  
\end{proof}

We now prove Theorem~\ref{thm:delta reg improvement}, the main result of this section. It establishes a weighted self-improvement for the $p$-Hardy inequalities in terms of $\delta$-regularizable weights which are $p$-admissible.  

\begin{proof}[Proof of Theorem~\ref{thm:delta reg improvement}]
    Let $w$ be $p$-admissible and $\delta$-regularizable in $\Omega$ for some $0<\delta\le\delta_p$, with $\delta_p$ to be chosen later.  Since $(X,d)$ is complete, and since $w\,d\mu$ is doubling and supports a $p$-Poincar\'e inequality, it follows from \cite{KeZ} that there exists $1\le q<p$ such that $w\,d\mu$ supports a $(1,q)$-Poincar\'e inequality,  with constants depending only on $p$-admissibility constants of $w$.  By choosing $q<p$ sufficiently close to $p$ such that $p/\max\{\log_2C_\mu,2p\}\ge p-q$, it follows from \cite[Theorem~4.21]{BB} for example, that $w\,d\mu$ also supports a $(q,p)$-Poincar\'e inequality with constants depending only on the $p$-admissibility constants of $w$.  Moreover, letting $\lambda\ge 1$ be the dilation constant of this $(1,q)$-Poincar\'e inequality, it follows that the $(p,q)$-Poincar\'e inequality has dilation constant $2\lambda$.      
    
    Let $\W_\Omega=\{B_i\}_i$ and $\{\phii_i\}_{i}$ be the Whitney cover of $\Omega$ and Lipschitz partition of unity given by Lemma~\ref{lem:Whitney Decomp}.  For each $i\in\N$, let 
    \[
    u_{B_i}:=\frac{1}{w(B_i)}\int_{B_i}uwd\mu,
    \]
    and define the discrete convolution $\wtil u$ of $u$ by 
    \[
    \wtil u=\sum_iu_{B_i}\phii_i.
    \]
    We then have that 
    \begin{align}\label{eq:3/28-1}
        \int_{\Omega}\frac{|u(x)|^p}{d_\Omega(x)^p}w(x)d\mu(x)\le 2^p\sum_{i}\int_{B_i}\frac{|u(x)-\wtil u(x)|^p}{d_\Omega(x)^p}w(x)d\mu(x)+2^p\sum_{i}\int_{B_i}\frac{|\wtil u(x)|^p}{d_\Omega(x)^p}w(x)d\mu(x).
    \end{align}
    
    To estimate the first term of the right-hand side of \eqref{eq:3/28-1}, we note that for $x\in B_i$, we have 
    \[
    u(x)-\wtil u(x)=u(x)-u_{B_i}+\sum_j(u_{B_j}-u_{B_i})\phii_j(x).
    \]
    Thus, it follows that 
    \begin{align*}
        |u(x)-\wtil u(x)|\le|u(x)-u_{B_i}|+\sum_{j:\,2B_j\cap B_i\ne\varnothing}|u_{B_j}-u_{B_i}|.
    \end{align*}
     We note that if $2B_j\cap B_i\ne\varnothing$, then $B_j\subset 6B_i$ and $B_i\subset 5B_j$. Since $w\,d\mu$ is doubling and supports a $(1,q)$-Poincar\'e inequality, it follows that 
    \begin{align}\label{eq:3/28-4}
    |u_{B_j}-u_{B_i}|\lesssim\frac{1}{w(6 B_i)}\int_{6B_i}|u-u_{6B_i}|wd\mu\lesssim r_i\left(\frac{1}{w(6\lambda B_i)}\int_{6\lambda B_i}g_u^qwd\mu\right)^{1/q}.
    \end{align}
    Hence, we have that 
    \[
    |u(x)-\wtil u(x)|\lesssim |u(x)-u_{B_i}|+r_i\left(\frac{1}{w(6\lambda B_i)}\int_{6 \lambda B_i}g_u^qw\,d\mu\right)^{1/q},
    \]
    where we have used bounded overlap of $\W_\Omega$. The comparison constant depends only on the doubling constant of $\mu$ and the $p$-admissibility constants of $w$. 
    Furthermore, as $w\,d\mu$ supports a $(p,q)$-Poincar\'e inequality,  and since $d_\Omega(x)\simeq r_i$ for each $x\in B_i$, we then obtain 
    \begin{align*}
        \int_{B_i}\frac{|u(x)-\wtil u(x)|^p}{d_\Omega(x)^p}&w(x)d\mu(x)\\
        &\lesssim \int_{B_i}\frac{|u(x)-u_{B_i}|^p}{d_\Omega(x)^p}w(x)d\mu(x)+r_i^p\left(\frac{1}{w(6\lambda B_i)}\int_{6\lambda B_i}g_u^qwd\mu\right)^{p/q}\int_{B_i}\frac{w(x)}{d_\Omega(x)^p}d\mu(x)\\
        &\lesssim w(B_i)\left(\frac{1}{w(6\lambda B_i)}\int_{6\lambda B_i}g_u^qwd\mu\right)^{p/q}\\
        &\le\int_{B_i}(M^w g_u^q)^{p/q}wd\mu.
        % &\lesssim \int_{ B_i}g_u^pwd\mu+\int_{6 B_i}g_u^pwd\mu.
    \end{align*}
Here $M^w$ is the (uncentered) Hardy-Littlewood maximal function, defined with respect to the measure $w\,d\mu$. As $w\,d\mu$ is doubling, $M^w$ is bounded from $L^{p/q}(X,wd\mu)$ to $L^{p/q}(X,wd\mu)$.  Using this fact and the bounded overlap of $\W_\Omega$, see Lemma~\ref{lem:Whitney Decomp}, we estimate the first term on the right-hand side of \eqref{eq:3/28-1} as follows:
\begin{align}\label{eq:3/28-2}
    \sum_{i}\int_{B_i}\frac{|u(x)-\wtil u(x)|^p}{d_\Omega(x)^p}w(x)d\mu(x)&\lesssim\sum_i \int_{B_i}(M^w g_u^q)^{p/q}wd\mu\nonumber\\
    &\lesssim\int_{\Omega}(M^w g_u^q)^{p/q}wd\mu\lesssim\int_\Omega g_u^pwd\mu.
\end{align}
Here the comparison constant depends only on $p$, $C_\mu$, and the $p$-admissibility constants of $w$. 

To estimate the second term on the right-hand side of \eqref{eq:3/28-1}, we see that 
\begin{align}\label{eq:3/28-3}
    \int_{B_i}\frac{|\wtil u(x)|^p}{d_\Omega(x)^p}w(x)d\mu(x)\lesssim\int_{B_i}\frac{|\wtil u(x)-u_{B_i}|^p}{d_\Omega(x)^p}w(x)d\mu(x)+\int_{B_i}\frac{|u_{B_i}|^p}{d_\Omega(x)^p}w(x)d\mu(x).
\end{align}
By using \eqref{eq:3/28-4}, and a similar argument to above, we estimate the first term on the right-hand side of the previous expression by 
\begin{align}\label{eq:3/28-5}
    \int_{B_i}\frac{|\wtil u(x)-u_{B_i}|^p}{d_\Omega(x)^p}w(x)d\mu(x)\lesssim\frac{1}{r_i^p}\sum_{j:\,2B_j\cap B_i\ne\varnothing}\int_{B_i}|u_{B_j}-u_{B_i}|^pwd\mu\lesssim\int_{B_i}(M^w g_u^q)^{p/q}wd\mu.
\end{align}

We note that $w(B_i)\simeq\wtil w(B_i)$, with comparison constant depending only on the doubling constant of $\mu$ coming from the Lipschitz partition of unity. Since $d_\Omega\simeq r_i$ on $B_i$, we then estimate the second term on the right-hand side of \eqref{eq:3/28-3} as follows:
\begin{align*}
    \int_{B_i}\frac{|u_{B_i}|^p}{d_\Omega(x)^p}w(x)d\mu(x)&\lesssim\int_{B_i}\frac{|u_{B_i}|^p}{d_\Omega(x)^p}\wtil w(x)d\mu(x)\\
    &\lesssim\int_{B_i}\frac{|\wtil u(x)-u_{B_i}|^p}{d_\Omega(x)^p}\wtil w(x)d\mu(x)+\int_{B_i}\frac{|\wtil u(x)|^p}{d_\Omega(x)^p}\wtil w(x)d\mu(x).
\end{align*}
The first term on the right-hand side of the above expression can be estimated using \eqref{eq:3/28-4} in a manner similar to \eqref{eq:3/28-5}, and so we then obtain
\[
\int_{B_i}\frac{|u_{B_i}|^p}{d_\Omega(x)^p}w(x)d\mu(x)\lesssim\int_{B_i}(M^w g_u^q)^{p/q}wd\mu+\int_{B_i}\frac{|\wtil u(x)|^p}{d_\Omega(x)^p}\wtil w(x)d\mu(x).
\] 
 Combining this inequality with \eqref{eq:3/28-3} and \eqref{eq:3/28-5}, we then have
\[
\int_{B_i}\frac{|\wtil u(x)|^p}{d_\Omega(x)^p}w(x)d\mu(x)\lesssim \int_{B_i}(M^w g_u^q)^{p/q}wd\mu+\int_{B_i}\frac{|\wtil u(x)|^p}{d_\Omega(x)^p}\wtil w(x)d\mu(x).
\]
Thus, by bounded overlap of $\W_\Omega$ and boundedness of $M^w$ from $L^{p/q}(X,wd\mu)$ to $L^{p/q}(X,wd\mu)$, we estimate the second term on the right-hand side of \eqref{eq:3/28-1} by 
\begin{align*}
\sum_{i}\int_{B_i}\frac{|\wtil u(x)|^p}{d_\Omega(x)^p}w(x)d\mu(x)\lesssim \int_{\Omega}g_u^pwd\mu+\int_{\Omega}\frac{|\wtil u(x)|^p}{d_\Omega(x)^p}\wtil w(x)d\mu(x).
\end{align*}
Combining this estimate with \eqref{eq:3/28-2} and \eqref{eq:3/28-1}, we then obtain
\begin{align}\label{eq:3/28-6}
   \int_{\Omega}\frac{|u(x)|^p}{d_\Omega(x)^p}w(x)d\mu(x)\lesssim \int_{\Omega}g_u^pwd\mu+\int_{\Omega}\frac{|\wtil u(x)|^p}{d_\Omega(x)^p}\wtil w(x)d\mu(x),
\end{align}
with comparison constants depending only on $p$, $C_\mu$, and the $p$-admissibility constants of $w$. 

Define the function $v:=|\wtil u|\wtil w^{1/p}$.  By the Leibniz rule and chain rule for upper gradients, see for example \cite[Theorem~2.15,\,Theorem~2.16]{BB}, and Lemma~\ref{lem:regularized gradient}, we then have that 
\[
g_v\le g_{\wtil u}\wtil w^{1/p}+\frac{1}{p}g_{\wtil w}\wtil w^{1/p-1}|\wtil u|\le g_{\wtil u}\wtil w^{1/p}+\frac{C_0}{p}\delta d_\Omega^{-1}\wtil w^{1/p}|\wtil u|
\]
where $C_0\ge 1$ is the constant from Lemma~\ref{lem:regularized gradient} which depends only on $C_\mu$. Applying the $p$-Hardy inequality to the function $v$ gives us 
\begin{align*}
    \int_{\Omega}\frac{|\wtil u(x)|^p}{d_\Omega(x)^p}\wtil w(x)d\mu(x)&=\int_{\Omega}\frac{|v(x)|^p}{d_\Omega(x)^p}d\mu(x)\\
    &\le C_p\int_\Omega g_v^p\,d\mu\le 2^pC_{p}\int_{\Omega}g_{\wtil u}^p\wtil wd\mu+2^pC_{p}\left(\frac{C_0\delta}{p}\right)^p\int_{\Omega}\frac{|\wtil u(x)|^p}{d_\Omega(x)^p}\wtil w(x)d\mu(x),
\end{align*}
where $C_p\ge 1$ is the constant from the $p$-Hardy inequality. Therefore, if $w$ is $\delta$-regularizable for $0<\delta\le\delta_p:=\frac{p}{2C_0}\left(2C_p\right)^{-1/p}$, we have that
\begin{equation*}
     \int_{\Omega}\frac{|\wtil u(x)|^p}{d_\Omega(x)^p}\wtil w(x)d\mu(x)\le 2^{p+1}C_p\int_{\Omega}g_{\wtil u}^p\wtil wd\mu,
\end{equation*}
and so from \eqref{eq:3/28-6} we obtain 
\begin{align}\label{eq:3/28-7}
    \int_{\Omega}\frac{|u(x)|^p}{d_\Omega(x)^p}w(x)d\mu(x)\lesssim\int_\Omega g_u^pwd\mu+\int_{\Omega}g_{\wtil u}^p\wtil wd\mu,
\end{align}
with comparison constant depending only on $p$, $C_\mu$, and the $p$-admissibility constants of $w$. 

Finally, for $x\in B_i$, we have that 
\[
\wtil u(x)=u_{B_i}+\sum_{j:\,2B_j\cap B_i\ne\varnothing}(u_{B_j}-u_{B_i})\phii_j(x),
\]
and so by  \eqref{eq:3/28-4}, we have
\[
g_{\wtil u}(x)\lesssim\frac{1}{r_i}\sum_{j:\,2B_j\cap B_i\ne\varnothing}|u_{B_j}-u_{B_i}|\lesssim\left(\frac{1}{w(6\lambda B_i)}\int_{6\lambda B_i}g_u^qwd\mu\right)^{1/q},
\]
where we have used the bounded overlap of $\W_\Omega$. 
Using the fact that $\wtil w(B_i)\simeq w(B_i)$, along with bounded overlap of $\W_\Omega$ and boundedness of $M^w$ on $L^{p/q}(X,wd\mu)$,  we then obtain
\begin{align*}
    \int_{\Omega}g_{\wtil u}^p\wtil wd\mu\le\sum_{i}\int_{B_i}g_{\wtil u}^p\wtil wd\mu&\lesssim\sum_{i}\wtil w(B_i)\left(\frac{1}{w(6\lambda B_i)}\int_{6\lambda B_i}g_u^qwd\mu\right)^{p/q}\\
    &\lesssim\sum_i\int_{B_i}(M^w g_u^q)^{p/q}wd\mu\lesssim\int_\Omega (M^w g_u^q)^{p/q}wd\mu\lesssim\int_\Omega g_u^pwd\mu.
\end{align*}
Here the constant depends only on $p$, $C_\mu$, and the $p$-admissibility constants of $w$.  Combining this estimate with \eqref{eq:3/28-7} completes the proof.
\end{proof}

\begin{remark}\label{rem:Bojarski}
   If in addition $(X,d)$ is assumed to be geodesic, then the above theorem can be proven in a slightly simpler manner.  If $w$ is $p$-admissible in this case, we can take $\lambda=1$ to be the dilation constant of the $(1,p)$-Poincar\'e inequality supported by the measure $w\,d\mu$, as shown in \cite{HaKo}. In this case, we do need not be concerned about the possible lack of bounded overlap of the collection $\{6\lambda B_i\}_i$, and so we can use the $(1,p)$-Poincar\'e inequality directly, rather than calling upon the deep Poincar\'e inequality self-improvement result of Keith and Zhong \cite{KeZ}.  To prove Theorem~\ref{thm:FracHardyImprovement} in Section~\ref{sec:Self-Improvement}, we will apply Theorem~\ref{thm:delta reg improvement} in the setting of a uniformized hyperbolic filling $(\oXeps,d_\eps,\mu_\beta)$ of a compact doubling metric measure space $(Z,d,\nu)$.  The space $(\oXeps,d_\eps,\mu_\beta)$ will be constructed such that it geodesic, see Section~\ref{sec:HypFill}.

\end{remark}

\begin{remark}\label{rem:Juhas weights}
    We see that for all $\delta>0$, there exists $\beta_0>0$ such that if $\beta\in\R$ with $|\beta|<\beta_0$, then $w:=d_\Omega^\beta$ is $\delta$-regularizable in $\Omega$.  Indeed, if $x,x'\in\Omega$ are such that $B(x',d_\Omega(x')/4)\cap B(x,d_\Omega(x)/8)\ne\varnothing$, then for all $y\in B(x,d_\Omega(x)/8))=:B_x$ and $z\in B(x',d_\Omega(x')/8))=:B_{x'}$, we have that $d_\Omega(y)\simeq d_\Omega(z)$. Thus, it follows that 
    \begin{align*}
    \left|w_{B_x}-w_{B_{x'}}
   \right|&\le\fint_{B_x}\fint_{B_{x'}}|d_\Omega(y)^\beta-d_\Omega(z)^\beta|d\mu(z)d\mu(y)\\
    &\le|1-C^\beta|\fint_{B_x}d_\Omega(y)^\beta d\mu(y)=|1-C^\beta|w_{B_x}
    % \frac{w(B_i)}{\mu(B_i)}
    \end{align*}
    for some $C>0$. Therefore, if $|\beta|$ is sufficiently small, then \eqref{eq:delta reg condition} is satisfied. 

    An examination of the proof of Theorem~\ref{thm:delta reg improvement} shows that the same conclusion holds if the $p$-admissibility assumption on $w$ is weakened to $q$-admissibility at Whitney scales in $\Omega$ for some $1<q<p$, see Definition~\ref{def:p-admissibile Whitney}. In the case that $X$ is geodesic, $w$ can be assumed to be $p$-admissible at Whitney scales in $\Omega$, by the same explanation as in Remark~\ref{rem:Bojarski}.  While weights of the form $d_\Omega^\beta$ may not be $p$-admissible in general without additional assumptions on $\Omega$, see \cite{BS} for example, we note that if $(X,d,\mu)$ supports a $(1,p)$-Poincar\'e inequality, then such weights are necessarily $q$-admissible at Whitney scales in $\Omega$ for some $1<q<p$.    From this, we see that Theorem~\ref{thm:delta reg improvement} recovers the self-improvement result \cite[Theorem~1]{L}.  As we show below in Proposition~\ref{prop:DistanceToPorousSets}, however, $\delta$-regularizable weights form a larger class, including powers of distance weights to porous sets, for example.  
\end{remark}

\begin{defn}\label{def:porous}
    Let $0<c\le 1$.  We say that a set $E\subset X$ is \emph{$c$-porous} if for all $x\in X$ and $0<r<\diam(E)$, there exists $y\in B(x,r)$ satisfying $B(y,cr)\subset X\setminus E$. 
\end{defn} 

The measure of the $\rho$-neighborhood of a porous set decays in a controlled manner, as shown by the following lemma: 

\begin{lem}{\cite[Theorem~2.8]{BS}}\label{lem:PorosityLayers}
If $E\subset X$ is $c$-porous, then there exist $\kappa>0$ and $C\ge 1$, depending only on $c$ and $C_\mu$, such that for all $x\in X$ and $0<\rho\le r<\diam(E)$, we have
\[
\mu\left(\{y\in B(x,r):d(y,E)<\rho\}\right)\le C\left(\frac{\rho}{r}\right)^\kappa\mu(B(x,r)).
\]    
\end{lem}

\begin{prop}\label{prop:DistanceToPorousSets}
    Let $\Omega\subset X$ be an open set and let $E\subset X$ be $c$-porous.  Then for all $0<\delta<1$, there exists $\sigma_0>0$, depending only $c$, $C_\mu$, and $\delta$, such that for all $\sigma\in \R$,  $\dist(\cdot, E)^\sigma$ is $\delta$-regularizable in $\Omega$ whenever $|\sigma|<\sigma_0$.  Furthermore, the dependence of $\sigma_0$ on $\delta$ is continuous.
\end{prop}

\begin{proof}
    Let $\delta>0$, and let $x,x'\in\Omega$ and $r,r'>0$ satisfy the conditions of Definition~\ref{def:reg weights}.  Let $B:=B(x,r)$ and $B':=B(x',r')$.  Let $\sigma\in\R$ and set $w:=d(\cdot, E)^\sigma$. We consider four cases, depending on the sign of $\sigma$ and the distance of $x$ to $E$ relative to $r$.
    \vskip.2cm
    \noindent\emph{\bf Case I:} Suppose that $\sigma\ge 0$ and $d(x,E)\ge 7r$.  Since $2B'\cap B\ne\varnothing$, it follows that $d(x',E)>2r'$, and so we have that
    \begin{align*}
        \left(\frac{6}{7}d(x,E)\right)^\sigma&\le w_{B}\le\left(\frac{8}{7}d(x,E)\right)^\sigma
    \end{align*}
    and also that
    \begin{align*}
        \left(\frac{1}{2}d(x',E)\right)^\sigma&\le w_{B'}\le\left(\frac{3}{2}d(x',E)\right)^\sigma.        
    \end{align*}
    Since $2B'\cap B\ne\varnothing$, it also follows that 
    \[
    \frac{3}{7}d(x,E)\le d(x',E)\le\frac{11}{7}d(x,E).
    \]
    Combining these estimates, we have that 
    \begin{align}
        \left|w_{B'}-w_{B}\right|&=\left|w_{B'}/w_{B}-1\right|w_{B}\nonumber\\
        &\le \max\left\{\left(\frac{11}{4}\right)^\sigma-1,1-\left(\frac{3}{16}\right)^\sigma\right\}w_{B}
    \end{align}
    Hence \eqref{eq:delta reg condition} holds in this case if 
    \begin{equation*}
    0\le|\sigma|<\min\left\{\frac{\log(\delta+1)}{\log(11/4)},\frac{\log(1/(1-\delta))}{\log(16/3)}\right\}=:\sigma_1.
    \end{equation*}
\vskip.2cm
\noindent{\bf Case II:}
    Suppose now that $\sigma\ge 0$ and $d(x,E)<7r$.  In this case, we also have that $d(x',E)<14r'$ since $2B'\cap B\ne\varnothing$.  Since $\sigma\ge0$, we have that 
    \[    w(B)\le(d(x,E)+r)^\sigma\mu(B)<(8r)^\sigma\mu(B).
    \]
    For $0<\rho<r$, we have by Lemma~\ref{lem:PorosityLayers} that
    \begin{align*}
        w(B)\ge w(B\setminus\{y\in X:d(y,E)<\rho\})&\ge\rho^\sigma\mu(B\setminus\{y\in X:d(y,E)<\rho\})\\
        &\ge\rho^\sigma\left(1-C(\rho/r)^\kappa\right)\mu(B),
    \end{align*}
    where $C\ge 1$ and $\kappa>0$ are the constants from Lemma~\ref{lem:PorosityLayers}.  Combining these estimates, we have that 
    \[
    \rho^\sigma\left(1-C(\rho/r)^\kappa\right)\le w_{B}\le(8r)^\sigma.
    \]
    By the same argument for $0<\rho<r'$, we also have 
    \[
    \rho^\sigma\left(1-C(\rho/r')^\kappa\right)\le w_{B'}\le(15r')^\sigma.
    \]
    Thus for $0<\rho<\min\{r,r'\}$, it follows that 
    \begin{align}\label{eq:6-2-24 Case I-ii}
      \left|w_{B'}-w_{B}\right|&=\left|w_{B'}/w_{B}-1\right|w_{B}\nonumber\\
      &\le \max\left\{\frac{(8r)^\sigma}{\rho^\sigma(1-C(\rho/r')^\kappa)}-1,\,1-\frac{\rho^\sigma(1-C(\rho/r)^\kappa)}{(15r')^\sigma}
      \right\}w_{B}.
      \end{align}

If $\rho=\tau_1 r$, where 
\[
\tau_1:=\frac{7}{10}\left(C^{-1}\left(1-\frac{1}{1+\delta/2}\right)\right)^{1/\kappa},
\]
then
\begin{align*}
\frac{(8r)^\sigma}{\rho^\sigma(1-C(\rho/r')^\kappa)}-1=\left(\frac{8}{\tau_1}\right)^\sigma\frac{1}{1-C(\tau_1r/r')^\kappa}-1&\le\left(\frac{8}{\tau_1}\right)^\sigma\frac{1}{1-C(10\tau_1/7)^\kappa}-1\\
	&=\left(\frac{8}{\tau_1}\right)^\sigma(1+\delta/2)-1.
\end{align*}
Likewise, if $\rho=\tau_2r'$, where 
\[
\tau_2:=\frac{2}{3}\left(\frac{\delta}{2C}\right)^{1/\kappa},
\]
we have that 
\begin{align*}
1-\frac{\rho^\sigma(1-C(\rho/r)^\kappa)}{(15r')^\sigma}=1-\left(\frac{\tau_2}{15}\right)^\sigma\left(1-C(\tau_2r'/r)^\kappa\right)&\le 1-\left(\frac{\tau_2}{15}\right)^\sigma\left(1-C(3\tau_2/2)^\kappa\right)\\
	&=1-\left(\frac{\tau_2}{15}\right)^\sigma\left(1-\delta/2\right).
\end{align*}
Hence, from these estimates and \eqref{eq:6-2-24 Case I-ii} it follows that \eqref{eq:delta reg condition} holds in this case if 
\begin{equation*}
0\le|\sigma|<\min\left\{\frac{\log((1+\delta)/(1+\delta/2))}{\log(8/\tau_1)},\,\frac{\log((1-\delta/2)/(1-\delta))}{\log(15/\tau_2)}\right\}=:\sigma_2.
\end{equation*}

\noindent{\bf Case III:}  Suppose now that $\sigma<0$ and $d(x,E)\ge 7r$.  In this case, by repeating the computations from Case I, we  obtain
\begin{align*}
        \left(\frac{8}{7}d(x,E)\right)^\sigma&\le w_{B}\le\left(\frac{6}{7}d(x,E)\right)^\sigma
    \end{align*}
   as well as
    \begin{align*}
        \left(\frac{3}{2}d(x',E)\right)^\sigma&\le w_{B'}\le\left(\frac{1}{2}d(x',E)\right)^\sigma.        
    \end{align*} 
We note that since $\sigma<0$, these are the reverse of the inequalities obtained in Case I.  From these estimates, it then follows that 
  \begin{align*}
        \left|w_{B'}-w_{B}\right|&=\left|w_{B'}/w_{B}-1\right|w_{B}\nonumber\\
        &\le \max\left\{1-\left(\frac{11}{4}\right)^\sigma,\left(\frac{3}{16}\right)^\sigma-1\right\}w_{B}.
    \end{align*}
Thus,  \eqref{eq:delta reg condition} holds in this case if 
\begin{align*}
0<|\sigma|<\min\left\{\frac{\log(1/(1-\delta))}{\log(11/4)},\,\frac{\log(1+\delta)}{\log(16/3)}\right\}=:\sigma_3.
\end{align*}

\noindent{\bf Case IV:} Suppose now that $\sigma<0$ and $d(x,E)<7r$.  In this case, we have that 
\begin{align*}
w(B)\ge (8r)^\sigma\mu(B),
\end{align*}  
and since $d(x',E)<14r'$, it follows that 
\begin{align*}
w(B')\ge(15r')^\sigma\mu(B').
\end{align*}

Let $0<\rho\le r$, and for each $k\in\N\cup\{0\}$, let 
\[
N_k:=\{y\in X:d(y,E)<2^{-k}\rho\}.
\]
We then have that 
\begin{align*}
w(B)&=\int_{B\cap N_0}d(\cdot,E)^\sigma d\mu+\int_{B\setminus N_0}d(\cdot,E)^\sigma d\mu\\
	&\le\sum_{k=0}^\infty\int_{B\cap N_k\setminus N_{k+1}}d(\cdot,E)^\sigma d\mu+\rho^\sigma\mu(B\setminus N_0)\\
	&\le\sum_{k=0}^\infty(2^{-k-1}\rho)^\sigma\mu(B\cap N_k)+\rho^\sigma\mu(B).
\end{align*}
By Lemma~\ref{lem:PorosityLayers}, we have that 
\[
\mu(B\cap N_k)\le C\left(\frac{2^{-k}\rho}{r}\right)^\kappa\mu(B),
\]
and so substituting this into the previous expression, we obtain
\begin{align*}
w(B)\le C2^{|\sigma|}\left(\frac{\rho}{r}\right)^\kappa\rho^\sigma\mu(B)\sum_{k=0}^\infty 2^{-k(\kappa-|\sigma|)}+\rho^\sigma\mu(B).
\end{align*}
Thus, for $0<|\sigma|<\kappa/2$, and setting 
\begin{equation*}
C_1:=C\frac{2^{\kappa/2}}{1-2^{-\kappa/2}}  ,
\end{equation*}
we have that 
\begin{align*}
w(B)\le\rho^\sigma\mu(B)\left(C_1\left(\frac{\rho}{r_i}\right)^\kappa+1\right).
\end{align*}
Recall that the $C\ge 1$ and $\kappa>0$ are the constants from Lemma~\ref{lem:PorosityLayers}.  For $0<\rho<r'$, the same computation, with $0<|\sigma|<\kappa/2$ gives us 
\[
w(B')\le\rho^\sigma\mu(B')\left(C_1\left(\frac{\rho}{r'}\right)^\kappa+1\right).
\]
Combining these two estimates with the corresponding lower bound estimates above, we obtain for $0<\rho\le\min\{r,r'\}$,
\begin{align*}
(8r)^\sigma\le w_{B}\le\rho^\sigma\left(C_1\left(\frac{\rho}{r}\right)^\kappa+1\right)
\end{align*}
and
\begin{align*}
(15r')^\sigma\le w_{B'}\le\rho^\sigma\left(C_1\left(\frac{\rho}{r'}\right)^\kappa+1\right).
\end{align*}
It then follows that 
\begin{align}\label{eq:Porosity Case IV}
 \left|w_{B'}-w_{B}\right|&=\left|w_{B'}/w_{B}-1\right|w_{B}\nonumber\\
        &\le\max\left\{\left(\frac{\rho}{8r}\right)^\sigma\left(C_1\left(\frac{\rho}{r'}\right)^\kappa+1\right)-1,\,1-\left(\frac{15r'}{\rho}\right)^\sigma\left(C_1\left(\frac{\rho}{r}\right)^\kappa+1\right)^{-1}\right\}w_{B}.
\end{align}

Let $\rho:=\min\{\tau_3r',\tau_4r\}$, where
\[
\tau_3:=\left(\frac{\delta}{2C_1}\right)^{1/\kappa}
\]
and 
\[
\tau_4:=\left(C_1^{-1}\left(\frac{1}{1-\delta/2}-1\right)\right)^{1/\kappa}.
\]
Note that $\rho\le\min\{r,r'\}$, and by 
% \eqref{eq:SimRadii}
the assumptions on $x$, $x'$, $r$, and $r'$, we also have that $\rho\ge\min\{7\tau_3/10,\tau_4\}r$, and $\rho\ge\min\{\tau_3,2\tau_4/3\}r'$. From these choices, it then follows that 
\begin{align*}
\left(\frac{\rho}{8r}\right)^\sigma\left(C_1\left(\frac{\rho}{r'}\right)^\kappa+1\right)-1\le\left(\frac{8}{\min\{7\tau_3/10,\tau_4\}}\right)^{|\sigma|}(1+\delta/2)-1.
\end{align*} 
Similarly, we have that 
\begin{align*}
1-\left(\frac{15r'}{\rho}\right)^\sigma\left(C_1\left(\frac{\rho}{r}\right)^\kappa+1\right)^{-1}\le 1-\left(\frac{\min\{\tau_3,2\tau_4/3\}}{15}\right)^{|\sigma|}(1-\delta/2).
\end{align*}
From \eqref{eq:Porosity Case IV}, we then see that \eqref{eq:delta reg condition} holds in this case if 
\begin{align*}
0<|\sigma|<\min\left\{\frac{\log((1+\delta)/(1+\delta/2))}{\log(8/\min\{7\tau_3/10,\tau_4\})},\frac{\log((1-\delta/2)/(1-\delta))}{\log(15/\min\{\tau_3,2\tau_4/3\})}\right\}=:\sigma_4.
\end{align*}

Having exhausted all cases, we complete the proof by taking $\sigma_0:=\min_{1\le i\le 4}\sigma_i$.  We note that in addition to depending on $\delta$, the constant $\sigma_0$ depends only on the constants $C$ and $\kappa$ from Lemma~\ref{lem:PorosityLayers}, which depend only on the porosity constant of $E$ and the doubling constant of $\mu$.  \qedhere

    \end{proof}

\section{The hyperbolic filling of a compact, doubling metric measure space}\label{sec:HypFill}
We now turn our attention towards the fractional Hardy inequality.  Recall from Definition~\ref{def:Frac Hardy} that for $0<\theta<1$ and $1<p<\infty$, the right-hand side of the $(\theta,p)$-Hardy inequality is given by the Besov energy \eqref{eq:Besov energy}.  The Besov spaces, defined by \eqref{eq:Besov energy}, arise naturally in both the Euclidean and metric settings as traces to the boundary of Sobolev functions defined on a domain, provided the domain satisfies certain geometric assumptions, see for example \cite{Gal,JW,Ma}. It was shown in \cite{BBS} that Besov spaces on a compact, doubling metric measure space $(Z,d,\nu)$ always arise as traces of Newton-Sobolev functions on an auxiliary metric measure space which has $Z$ as a boundary.  We  now outline the construction from \cite{BBS}, as well as some properties of the resulting space, which will allow us in the subsequent sections to relate $(\theta,p)$- and $p$-Hardy inequalities.

Throughout this section, we let $(Z,d,\nu)$ be a compact, doubling metric measure space and assume that $\diam(Z)<1$.  Fix $\alpha,\tau>1$, and fix $z_0\in Z$.  Let $A_0:=\{z_0\}$, and for each $i\in\N$, let $A_i$ be a maximally $\alpha^{-i}$-separated subset of $Z$, constructed inductively so that $A_i\subset A_{i+1}$. For each $i\in\N\cup\{0\}$, we define the $i$-th level vertices by 
\[
V_i:=\bigcup_{z\in A_i}(z,i),
\]
and define the vertex set
\[
V=\bigcup_{i=0}^\infty\bigcup_{z\in A_i}(z,i).
\]
Given $(z,i),(y,j)\in V$, we say that $(z,i)$ and $(y,j)$ are neighbors, denoted $(z,i)\sim(y,j)$, if and only if $|i-j|\le 1$ and in addition, either $i=j$ and $B_Z(z,\tau\alpha^{-i})\cap B_Z(y,\tau\alpha^{-j})\ne\varnothing$, or $i=j\pm 1$ and $B_Z(z,\alpha^{-i})\cap B_Z(y,\alpha^{-j})\ne\varnothing$. We then turn the vertex set $V$ into a metric graph $X$ by gluing a unit interval $[v,w]$ between each pair of neighboring vertices $v,w$.  We denote by $d_X$ the path metric on $X$. 

For $\eps>0$, we consider the uniformized metric on $X$, given by 
\begin{equation*}
    d_\eps(x,y):=\inf_\gamma\int_\gamma e^{-\eps d_X(\cdot,v_0)}ds,
\end{equation*}
where the infimum is taken over all paths $\gamma$ in $X$ with endpoints $x$ and $y$. 
We denote by $\oXeps$ the completion of $X$ with respect to the metric $d_\eps$, and we set $\partial_\eps X:=\oXeps\setminus X$. We refer to $(\oXeps,d_\eps)$ as the \emph{uniformized hyperbolic filling} of $(Z,d,\nu)$.  For clarity, balls in $Z$, taken with respect to the $d$, are denoted by $B_Z$, while balls in $\oXeps$ taken with respect to $d_\eps$ are denoted by $B_\eps$. 

For $\beta>0$, we define the vertex weights $\hat\mu_\beta$ by
\begin{equation*}
    \hat\mu_\beta((z,i))=e^{-\beta i}\nu(B_Z(z,\alpha^{-i})).
\end{equation*}
We then define the measure $\mu_\beta$ on $X$ by
\begin{equation*}
    \mu_\beta(A):=\sum_{v\in V}\sum_{w\sim v}\left(\hat\mu_\beta(v)+\hat\mu_\beta(w)\right)\Ha^{1}(A\cap[v,w]).
\end{equation*}
From this definition, we see that if $\sigma\in\R$ is such that $|\sigma|<\beta$, then 
\begin{equation}\label{eq:WeightedMeasure}
\mu_{\beta+\sigma}\simeq d_\eps(\cdot,Z)^{\sigma/\eps}d\mu_\beta.
\end{equation} 
That is, perturbing the parameter $\beta$ is equivalent to weighting the measure $\mu_\beta$, see \cite[Section~10]{BBS}.  The following summarizes the fundamental properties of $(\oXeps,d_\eps,\mu_\beta)$, obtained in \cite{BBS}:

\begin{thm}\cite[Theorem~1.1]{BBS}\label{thm:HypFill}
    Let $(Z,d,\nu)$ be a compact metric measure space, $\diam(Z)<1$, with $\nu$ a doubling measure. Let $\alpha,\tau>1$ and let $\eps:=\log\alpha$.  Then for all $\beta>0$, $(\overline X_\eps,d_\eps,\mu_\beta)$ satisfies the following:
    \begin{enumerate}
        \item $(\oXeps,d_\eps)$ is a geodesic space.
        \item $(Z,d)$ is bi-Lipschitz equivalent to $(\partial_\eps X,d_\eps)$, with bi-Lipschitz constant depending only on $\alpha$ and $\tau$. 
        \item $(X,d_\eps,\mu_\beta)$ and $(\overline X_\eps,d_\eps,\mu_\beta)$ are doubling and support a $(1,1)$-Poincar\'e inequality. 
        \item For all $z\in Z$ and $0<r<2\diam (Z)$, we have that 
        \[
        \nu(B_\eps(z,r)\cap Z)\simeq\frac{\mu_\beta(B_\eps(z,r))}{r^{\beta/\eps}}.
        \]
        \item If $1\le p<\infty$ and $0<\theta<1$ are such that $\beta/\eps=p(1-\theta)$, then there exist bounded trace and extension operators
        \[
        T:N^{1,p}(\overline X_\eps,\mu_\beta)\to B^\theta_{p,p}(Z,\nu)\quad and \quad E:B^\theta_{p,p}(Z,\nu)\to N^{1,p}(\overline X_\eps,\mu_\beta),
        \] 
        such that $T\circ E=Id$.  Moreover, for all $u\in N^{1,p}(\oXeps,\mu_\beta)$ and $f\in B^\theta_{p,p}(Z,\nu)$, we have
        \begin{equation*}
            \|Tu\|^p_{B^\theta_{p,p}(Z,\nu)}\lesssim\int_{\oXeps}g_u^p\,d\mu_\beta,\qquad \int_{\oXeps}g_{Ef}^p\,d\mu_\beta\lesssim\|f\|^p_{B^\theta_{p,p}(Z,\nu)}.
        \end{equation*}
    \end{enumerate}
    In the above results, the constants depend only on $\alpha$, $\tau$, $\beta$, and $C_\nu$.
\end{thm}
The trace and extensions operators above have the following explicit constructions: for $u\in N^{1,p}(\oXeps,\mu_\beta)$ and $z\in Z$, we have 
\begin{equation*}
    Tu(z)=\lim_{r\to 0^+}\fint_{B_\eps(z,r)}u\,d\mu_\beta.
\end{equation*}
Likewise, for $f\in B^\theta_{p,p}(Z,\nu)$, $Ef$ is first defined on the vertices of $X$ by
\begin{equation}\label{eq:HypFillExtension}
    Ef((z,i))=\fint_{B_Z(z,\alpha^{-i})}f\,d\nu.
\end{equation}
To define $Ef$ on all of $X$, these values are then extended piece-wise linearly (with respect to $d_\eps$) to the edges of $X$.   

Balls centered at vertices of $\oXeps$ form a Whitney-type covering of $X$, as shown in the following lemma.  We first note that if $\alpha>1$ and $\eps=\log\alpha$, then for all $i\in\N\cup\{0\}$ and $z\in A_i$, we have that
\begin{equation}\label{eq:VertexToBoundary}    d_\eps((z,i),Z)=d_\eps((z,i),z)=\frac{\alpha^{-i}}{\log\alpha}.
\end{equation}
This can be shown by direct computation using the definition of $d_\eps$, along with the fact that $(z,i+j)\in V$ for all $j\in\N$, see \cite[Section~4]{BBS}.
\begin{lem}\label{lem:HypFillWhitney}
    Let $\alpha=e^{1/4}$, and let $\eps:=\log\alpha$. Then the following hold:
    \begin{enumerate}
        \item For all $i\in\N\cup\{0\}$ and $z\in A_i$, we have that $d_\eps((z,i),Z)=4\alpha^{-i}$.\label{eq:WhitneyCover 0}
        \item $X=\bigcup_{i=0}^\infty\bigcup_{z\in A_i}B_\eps((z,i),\alpha^{-i})$.\label{eq:WhitneyCover 1}
        \item $\sum_{i=0}^\infty\sum_{z\in A_i}\chi_{ B_\eps\left((z,i),\alpha^{-i}\right)}\le C:=C(\alpha,\tau,C_\nu)$.\label{eq:WhitneyCover 2}
    \end{enumerate}
   \end{lem}

\begin{proof}
Claim \eqref{eq:WhitneyCover 0} follows directly from \eqref{eq:VertexToBoundary}. To prove Claim \eqref{eq:WhitneyCover 1}, let $v,w\in V$ be such that $v\sim w$.  Without loss of generality, we may assume that $v=(z,i)$ and $w=(y,j)$ with $j\in\{i,i+1\}$.  If $j=i$, then by the definition of $d_\eps$ along with \cite[Proposition~4.4]{BBS}, we have that 
\begin{align*}
\frac{\alpha^{-i}}{2\tau\alpha}\le\frac{d_Z(z,y)}{2\tau\alpha}\le d_\eps(z,w)\le\ell_\eps([v,w])=\alpha^{-i}.
\end{align*}
If $j=i+1$, then we have from the definition of $d_\eps$ and choice of $\alpha=e^{1/4}$ that 
\[
d_\eps(v,w)=\ell_\eps([v,w])=\alpha^{-i}\left(\frac{\alpha-1}{\alpha\log\alpha}\right)<\alpha^{-i}.
\]
In either case, it follows that $[v,w]\subset B_\eps(v,\alpha^{-i})\cup B_\eps(w,\alpha^{-j})$.  By our choice of $\alpha=e^{1/4}$, it follows from \eqref{eq:VertexToBoundary} that $\dist_\eps(v,Z)=4\alpha^{-i}$.  Thus, $B_\eps(v,\alpha^{-i})\cap Z=\varnothing$, and so \eqref{eq:WhitneyCover 1} holds.

To prove \eqref{eq:WhitneyCover 2}, we first claim that if $v=(z,i)$ and $w=(y,j)$ are vertices such that $j=i+k$ for some $k\ge 3$, then $B_\eps(v,\alpha^{-i})\cap B_\eps(w,\alpha^{-j})=\varnothing$.  Indeed, if $\gamma$ is a path in $X$ joining $v$ to $w$, then $\gamma$ must contain a vertical edge joining the $(i+l-1)$-th level of vertices to the $(i+l)$-th level of vertices for each $1\le l\le k$.  Thus, it follows that 
\[
d_\eps(v,w)\ge \int_{i}^{i+k}e^{-\eps t}dt=\frac{\alpha^{-i}(1-\alpha^{-k})}{\log\alpha}.
\]
Since $k\ge 3$, and since $\alpha=e^{1/4}$, it follows that 
\[
d_\eps(v,w)\ge\frac{\alpha^{-i}(1-\alpha^{-k})}{\log\alpha}>\alpha^{-i}+\alpha^{-j},
\]
hence $B_\eps(v,\alpha^{-i})\cap B_\eps(w,\alpha^{-j})=\varnothing$, and the claim follows.

Now, let $x\in X$.  Then, there exists $i\in\N\cup\{0\}$ and $z\in A_i$ such that $x\in B_\eps((z,i),\alpha^{-i})$.  If $j\in\N\cup\{0\}$ and $y\in A_j$ are such that $x\in B_\eps((y,j),\alpha^{-j})$, then we know from the above claim that $|i-j|\le 3$. If $z\ne y$, then since $z,y\in A_{i+3}$, it follows from \cite[Proposition~4.4]{BBS} that 
\[
\alpha^{-(i+3)}\le d_Z(z,y)\le 2\tau\alpha d_\eps((z,i),(y,j))\le 2\tau\alpha(\alpha^{-i}+\alpha^{-j})\le 2\tau\alpha(1+\alpha^3)\alpha^{-i}.
\]
Therefore, since $d_Z$ is doubling, there exists a constant $C:=C(\alpha,\tau,C_\nu,\diam(Z))$ such that the set 
\[
\{y\in A_{i+3}:\exists j\in\{i-3,\dots,i+3\}\st x\in B((y,j),\alpha^{-j})\}
\]
contains at most $C$ elements.  For each $y$ in this set, the only possible vertices $(y,j)$ for which $B_\eps((y,j),\alpha^{-j})$ may contain $x$ are $\{(y,i-3),\,(y,i-2),\dots,\,(y,i+3)\}$.  Thus, it follows that the set 
\[
\{(y,j): j\in\N\cup\{0\},\,y\in A_j,\, x\in B_\eps((y,j),\alpha^{-j})\}
\] 
contains at most $7C$ elements, and so \eqref{eq:WhitneyCover 2} follows. 
\end{proof}

Recall the definition of a porous set given in Definition~\ref{def:porous}.  
\begin{lem}\label{lem:Z porous in X}
Let $\alpha,\tau>1$, and let $\eps:=\log\alpha$.  Then there exists $0<c<1$, depending only on $\alpha$, such that $Z\subset\oXeps$ is $c$-porous.  
\end{lem}

\begin{proof}
Let $z\in Z$ and let $0<r\le\diam(Z)$.  Let $K:=2(1+1/\log\alpha)$.  Let $i\in\N\cup\{0\}$ be the smallest non-negative integer such that $\alpha^{-i}<r/K$.  Then there exists $z'\in A_{i}\cap B_\eps(z,r/K)$.  Let $v:=(z',i)\in V$.  We then have from \eqref{eq:VertexToBoundary} that 
\begin{align*}
d_\eps(z,v)\le d_\eps(z,z')+d_\eps(v,z')<\frac{r}{K}+\frac{\alpha^{-i}}{\log\alpha}<\frac{r}{K}(1+1/\log\alpha)=r/2.
\end{align*}
Thus, $v\in B_\eps(z,r/2)$.  Since $d_\eps(v,Z)=\alpha^{-i}/\log\alpha$ from \eqref{eq:VertexToBoundary}, it follows that $B_\eps(v,\alpha^{-i}/(2\log\alpha))\subset\oXeps\setminus Z$.  By our choice of $i$, we then have that 
\[
B_\eps(v,r/(4\alpha(1+\log\alpha)))\subset B_\eps(v,\alpha^{-i}/(2\log\alpha))\subset\oXeps\setminus Z.
\]
Thus, $Z$ is $c$-porous in $\oXeps$ with $c:=(4\alpha(1+\log\alpha))^{-1}$.
\end{proof}

\begin{remark}\label{rem:diam(Z)}
    The assumption that $\diam(Z)<1$ is not overly restrictive.  If $\diam(Z)\ge 1$, then we may rescale the metric, replacing $d$ with $d/(2\diam(Z))$.  In this case, analogs of the above results hold for $(Z,d,\nu)$; the only difference is that the constants present in (some of) the above results will depend additionally upon $\diam(Z)$. 
\end{remark}

\section{Fractional Hardy inequalities and the hyperbolic filling}\label{sec:FracHardy-HypFill}

In this section, we assume that $(Z,d,\nu)$ is a compact metric measure space, with $\nu$ a doubling measure.  By rescaling the metric $d$ if necessary, we may assume without loss of generality that $\diam(Z)<1$, see Remark~\ref{rem:diam(Z)}.  We fix parameters $\alpha=e^{1/4}$, $\tau=2$, and we let $\eps:=\log\alpha=1/4$.  For each $\beta>0$, we consider the uniformized hyperbolic filling $(\overline X_\eps,d_\eps,\mu_\beta)$ of $(Z,d,\nu)$ as constructed in Section~\ref{sec:HypFill}. The main result of this section is the following, which relates the validity of $(\theta,p)$-Hardy inequalities on $Z$ to $p$-Hardy inequalities on $\overline X_\eps$. This theorem is proven by combining Proposition~\ref{prop:ZtoX} and Proposition~\ref{prop:XtoZ} below.

In what follows, we denumerate the $\alpha^{-i}$-separated set $A_i$ used in the construction of $(\oXeps,d_\eps,\mu_\beta)$ by 
\[
A_i:=\{z_{i,j}\}_{j\in I_i}.
\]
For ease of notation, we set
\[
I:=\{(i,j):i\in\N\cup\{0\},\,j\in I_i\},
\]
and for each $(i,j)\in I$, we set 
\[
B_{i,j}:=B_\eps((z_{i,j},i),\alpha^{-i}).
\]
By Lemma~\ref{lem:HypFillWhitney}, the collection $\{B_{i,j}\}_{(i,j)\in I}$ covers $X$ and has bounded overlap.  For each $(i,j)\in I$, we denote by $U_{i,j}$ the set 
\[
U_{i,j}:=B_Z(z_{i,j},\alpha^{-i})\cap Z.
\]
Recall that we denote balls in $Z$ taken with respect to the metric $d$ by $B_Z$, while balls in $\oXeps$, taken with respect to $d_\eps$, are denoted by $B_\eps$. For a constant $C>0$ we also denote by $CU_{i,j}$ the set $B_Z(z_{i,j},C\alpha^{-i})\cap Z$.  We now show that given a closed set $E\subset Z$, there exists a subcollection of $\{U_{i,j}\}_{(i,j)\in I}$ which forms a Whitney-type cover of $Z\setminus E$.

\begin{lem}\label{lem:InducedCover}
Let $E\subset Z$ be a closed set.  Then there exists $I_E\subset I$ such that the following hold:
\begin{enumerate}
    \item $Z\setminus E=\bigcup_{(i,j)\in I_E} U_{i,j}$,\label{eq:U-cover}
    \item For each $(i,j)\in I_E$, we have that $6\alpha^{-i}\le d(U_{i,j},E)\le(8\alpha)\alpha^{-i}$,\label{eq:U-distance}
    \item $\sum_{(i,j)\in I_E}\chi_{3U_{i,j}}\le C:=C(\alpha,C_\nu)$.\label{eq:U-overlap} 
\end{enumerate}
\end{lem}

\begin{proof}
For each $i\in\N\cup\{0\}$, let 
\begin{equation}\label{eq:Neighborhood}
N_i:=\{z\in Z\setminus E:8\alpha^{-i}\le d(z,E)<8\alpha^{-(i-1)}\},
\end{equation}
and let 
\[
I_{i,E}:=\{j\in I_i:U_{i,j}\cap N_i\ne\varnothing\}.
\]
We then define
\[
I_E:=\bigcup_{i=0}^\infty\{(i,j)\in I:j\in I_{i,E}\}.
\]
For $z\in Z\setminus E$, there exists a unique $i\in\N\cup\{0\}$ such that $z\in N_i$.  As $A_i$ is a maximally $\alpha^{-i}$-separated subset of $Z$, there $z_{i,j}\in A_i$ such that $x\in B_Z(z_{i,j},\alpha^{-i})=U_{i,j}$.  Thus, $(i,j)\in I_{i,E}$, and so it follows that $Z\setminus E\subset\bigcup_{(i,j)\in I_E}U_{i,j}$.

For $(i,j)\in I_E$, we have that there exists $y\in U_{i,j}\cap N_i$.  For any $z\in U_{i,j}$, it follows from the triangle inequality that $d(z,E)\ge d(y,E)-d(y,z)\ge 8\alpha^{-i}-2\alpha^{-i}=6\alpha^{-i}$,
and so we have that 
\begin{equation}\label{eq:U dist to E}
d(U_{i,j},E)\ge 6\alpha^{-i}.
\end{equation}
This completes the proof of \eqref{eq:U-cover}, as well as the first inequality of \eqref{eq:U-distance}.
Since $y\in U_{i,j}\cap N_i$, we also have that 
\[
d(U_{i,j},E)\le d(y,E)\le (8\alpha)\alpha^{-i},
\]
which gives us the second inequality in \eqref{eq:U-distance}.

Now, let $(i,j)\in I_E$, and let $z\in 3U_{i,j}$.  Since there exists $y\in U_{i,j}\cap N_i$, it follows from the triangle inequality that \[
d(z,E)\le d(z,y)+d(y,E)<(6+8\alpha)\alpha^{-i}\]
and also that \[
d(z,E)\ge d(y,E)-d(y,z)> 2\alpha^{-i}.\]  Hence we have that 
\[
2\alpha^{-i}<d(3U_{i,j},E)<(6+8\alpha)\alpha^{-i}.
\]
From this and the choice of $\alpha=e^{1/4}$, it follows that if $l\in\N$ is such that $l>7$, then $3U_{i,j}\cap N_{i+l}=\varnothing$, and if $l>5$, then $3U_{i,j}\cap N_{i-l}=\varnothing$. Hence, $3U_{i,j}\subset\bigcup_{l=-5}^{7}N_{i+l}$.
From this, we see that if $i_1,i_2\in \N\cup\{0\}$ are such that $|i_1-i_2|>12$, then 
\begin{equation}\label{eq:12distance}
    \left(\bigcup_{j\in I_{i_1,E}}3U_{i_1,j}\right)\cap\left(\bigcup_{j\in I_{i_2,E}}3U_{i_2,j}\right)=\varnothing.
\end{equation}

Let $\zeta\in Z\setminus E$. There is a unique $i_0\in\N\cup\{0\}$ such that $\zeta\in N_{i_0}$, and so from \eqref{eq:12distance}, it follows that 
\begin{align*}
    \sum_{(i,j)\in I_E}\chi_{3U_{i,j}}(\zeta)=\sum_{i=i_0-12}^{i_0+12}\sum_{j\in I_{i,E}}\chi_{3U_{i,j}}(\zeta).
\end{align*}
Since $\{z_{i,j}\}_{j\in I_{i,E}}$ is an $\alpha^{-i}$-separated set and $\nu$ is doubling, there exists $C:=C(\alpha,C_\nu)$ such that $\sum_{j\in I_{i,E}}\chi_{3U_{i,j}}\le C$.  Therefore, substituting this into the previous expression, we have that 
\[
 \sum_{(i,j)\in I_E}\chi_{3U_{i,j}}(\zeta)\le 24 C,
\]
which completes the proof of \eqref{eq:U-overlap}.\qedhere

\end{proof}

Given $(i,j)\in I$, the structure of the hyperbolic filling also gives us a natural collection of balls with which to chain $B_{i,j}$ to $U_{i,j}$, as we show with the following lemma.

\begin{lem}\label{lem:Chain}
    Let $(i,j)\in I$, and let $z\in U_{i,j}$.  For each $k\in \N\cup\{0\}$, $k\ge i$, there exists $j_k\in I_k$ such that the collection $\{B_k(z):=B_{k,j_k}\}_{k=i}^\infty$ satisfies the following properties for all $k\ge i$:
    \begin{enumerate}
        \item $B_i(z)=B_{i,j}$,\label{eq:Chain 1}
        \item $\rad(B_k(z))=\alpha^{-k}\simeq d_\eps(z,B_k(z))$,\label{eq:Chain 2}
        \item $B_{k+1}(z)\subset 2B_k(z)$ for all $k\ge i$,\label{eq:Chain 3}
        \item $B_k(z)\subset 5B_{i,j}$,\label{eq:Chain 4}
        \item $\mu_\beta(B_k(z))\simeq\alpha^{-k\beta/\eps}\nu(B_\eps(z,\alpha^{-k}))$,\label{eq:Chain 5}
        \item If $u\in\Lip(\oXeps)$, then $Tu(z)=\lim_{k\to\infty}u_{B_k(z)}$.\label{eq:Chain 6}
    \end{enumerate}
    Here $T$ is the trace operator given in Theorem~\ref{thm:HypFill}.
\end{lem}

\begin{proof}
    Set $j_i:= j$, and for all $k\in\N$, $k>i$, choose $j_k\in I_k$ such that $d(z,z_{k,j_k})<\alpha^{-k}$.  Such a $j_k$ exists, as $A_k$ is a maximally $\alpha^{-k}$-separated set.  Setting $v_k:=(z_{k,j_k},k)\in V$, it follows from the definition of the neighborhood relationship between vertices of $X$, given in Section~\ref{sec:HypFill}, that $v_k\sim v_{k+1}$ for all $k\ge i$. 

    For each $k\ge i$, set
    \[    B_k(z):=B_{k,j_k}=B_\eps(z_{k,j_k},\alpha^{-k}).
    \]
    By \eqref{eq:VertexToBoundary}, it follows that
    \[
    d_\eps(z,B_k(z))\ge d_\eps(B_k(z),Z)\ge 3\alpha^{-k}.
    \]
    Likewise, by the bi-Lipschitz equivalence of $d$ and $d_\eps$, as well as \eqref{eq:VertexToBoundary}, we have
    \[
    d_\eps(z, B_k(z))\le d_\eps(z,z_{k,j_k})+d_\eps(z_{k,j_k},v_k)\simeq d(z,z_{k,j_k})+\alpha^{-k}\lesssim\alpha^{-k}.
    \]
    Thus, \eqref{eq:Chain 1} and \eqref{eq:Chain 2} hold.

    Since $v_k\sim v_{k+1}$, it follows from the definition of $d_\eps$ and choice of $\alpha=e^{1/4}$ that 
    \[
    d_\eps(v_k,v_{k+1})\le\ell_\eps([v_k,v_{k+1}])=\int_{k}^{k+1}e^{-\eps t}\,dt=4(1-1/\alpha)\alpha^{-k}<\alpha^{-k}.
    \]
    Thus, $v_{k+1}\in B_k(z)$, and so we have have that $B_{k+1}(z)\subset 2B_k(z)$, proving \eqref{eq:Chain 3}.

    If $x\in B_k(z)$, then by \eqref{eq:VertexToBoundary}, we have that 
    \[
    d_\eps(v_i,x)\le d_\eps(v_i,v_k)+d_\eps(v_k,x)<4\alpha^{-i}+\alpha^{-k}\le 5\alpha^{-i},
    \]
    and so we have that $B_k(z)\subset B_{i,j}$, proving \eqref{eq:Chain 4}.

    To prove \eqref{eq:Chain 5}, we note that by the bi-Lipschitz equivalence of $d$ and $d_\eps$, there is a constant $C:=C(\alpha,\tau)\ge 1$ such that $d_\eps(z,z_{k,j_k})<C\alpha^{-k}$. From \eqref{eq:VertexToBoundary} and our choice of $\alpha=e^{1/4}$, it then follows that 
    \[
    B_k(z)\subset B_\eps(z_{k,j_k},5\alpha^{-k})\subset B_\eps(z, (5+C)\alpha^{-k}).
    \]
    Since $\mu_\beta$ is doubling, we then have that 
    \[
    \mu_\beta(B_k(z))\le \mu_\beta(B_\eps(z,(5+C)\alpha^{-k})\simeq\mu_\beta(B_\eps(z,\alpha^{-k})).
    \]
    Likewise, by \eqref{eq:VertexToBoundary}, we have that $B_\eps(z,\alpha^{-k})\subset B(z_{k,j_k},(1+C)\alpha^{-k})\subset(5+C)B_k(z)$, 
    and so again by the doubling property of $\mu_\beta$, we obtain
    \[
    \mu_\beta(B_\eps(z,\alpha^{-k}))\le\mu_\beta((5+C)B_k(z))\simeq\mu_\beta(B_k(z)).
    \]
    From these two inequalities and the $\beta/\eps$-codimensional relationship between $\mu_\beta$ and $\nu$ given by Theorem~\ref{thm:HypFill}, we have that 
    \[
    \mu_\beta(B_k(z))\simeq\alpha^{-k\beta/\eps}\nu(B_\eps(z,\alpha^{-k})),
    \]
    proving \eqref{eq:Chain 5}.

    Finally, if $u\in\Lip(\oXeps)$, then $Tu$ is given by the restriction of $u$ to $Z$, and is thus also Lipschitz.  Moreover, for each $x\in B_k(z)$, it follows that \[
    d_\eps(x,z)\le d_\eps(x,v_k)+d_\eps(v_k,z_{k,j_k})+d_\eps(z_{k,j_k},z)<6C\alpha^{-k}.
    \]
    Hence, 
    \[
    |Tu(z)-u_{B_k(z)}|\le\fint_{B_k(z)}|u(z)-u(x)|d\mu_\beta(x)\lesssim\alpha^{-k}\to 0
    \]
    as $k\to\infty$.  This proves \eqref{eq:Chain 6}.    
\end{proof}

We now show that a $(\theta,p)$-Hardy inequality on $Z$ implies a $p$-Hardy inequality with respect to $\mu_\beta$ on $\oXeps$, for an appropriate choice of $\theta$, $p$, and $\beta$.

\begin{prop}\label{prop:ZtoX}
Let $E\subset Z$ be a closed set, let $1<p<\infty$, $0<\theta<1$, and $\beta>0$ be such that $\beta/\eps=p(1-\theta)$, and suppose that $Z\setminus E$ satisfies a  $(\theta,p)$-Hardy inequality.  Then $\oXeps\setminus E$ satisfies a $p$-Hardy inequality with respect to $\mu_\beta$, with constant depending only on $\theta$, $p$, $C_\nu$, and $C_{\theta,p}$.   
\end{prop}

\begin{proof}
Let $u\in\Lip_c(\oXeps\setminus E)$, and for ease of notation, let $\mu:=\mu_\beta$.  For $(i,j)\in I$, let $B_{i,j}$ and $U_{i,j}$ be as defined above.  
 From Lemma~\ref{lem:HypFillWhitney}, we then have that 
\begin{align}\label{eq:512-3}
    \int_{\oXeps\setminus E}\frac{|u(x)|^p}{d_\eps(x,E)^p}d\mu(x)&\le\sum_{(i,j)\in I}\int_{B_{i,j}}\frac{|u(x)|^p}{d_\eps(x,E)^p}d\mu(x)\nonumber\\
    &\lesssim\sum_{(i,j)\in I}|(Tu)_{U_{i,j}}|^p\int_{B_{i,j}}\frac{d\mu(x)}{d_\eps(x,E)^p}+\sum_{(i,j)\in I}\int_{B_{i,j}}\frac{|u(x)-(Tu)_{U_{i,j}}|^p}{d_\eps(x,E)^p}d\mu(x),
\end{align}
where $T$ is the trace operator given by Theorem~\ref{thm:HypFill}.

To estimate the first term of \eqref{eq:512-3}, we decompose the collection $\{B_{i,j}\}_{(i,j)\in I}$ in the following manner.  Let $I_E\subset I$ be the indexing set given by Lemma~\ref{lem:InducedCover}, and for each $(i,j)\in I_E$, let \[
I_{i,j}:=\{(k,l)\in I:k>i,\, l\in I_{k},\,U_{k,l}\cap U_{i,j}\ne\varnothing\}.\]
Let $I_1:=\bigcup_{(i,j)\in I_E}I_{i,j}$, and let $I_2:=I\setminus I_1$.

For $(k,l)\in I_1$, we have from the bi-Lipschitz equivalence of $d$ and $d_\eps$, along with \eqref{eq:VertexToBoundary} that
\begin{align}\label{eq:bi-Lip 1}
    d(U_{k,l},E)&\simeq d_\eps(U_{k,l},E)\\
    &\le d_\eps(U_{k,l},B_{k,l})+d_\eps(B_{k,l},E)\lesssim d_\eps(B_{k,l},Z)+d_\eps(B_{k,l},E)\nonumber\lesssim d_\eps(B_{k,l},E).
\end{align}
Here the comparison constants depend only the bi-Lipschitz constants between $d$ and $d_\eps$, which depend only on $\alpha$ and $\tau$, see \cite[Proposition~4.4]{BBS}.  Moreover, if $(k,l)\in I_1$, then there exists $(i,j)\in I_E$ such that $(k,l)\in I_{i,j}$, and so there exists $y\in U_{k,l}\cap U_{i,j}$.  Thus, for $z\in U_{k,l}$, it follows from \eqref{eq:U dist to E} and Lemma~\ref{lem:InducedCover} \eqref{eq:U-distance} that 
\begin{align}\label{eq:d(z,E)}
    d(z,E)\ge d(U_{i,j},E)-d(y,z)\ge 6\alpha^{-i}-2\alpha^{-k}\ge 4\alpha^{-k}\ge (2\alpha)^{-1}d(U_{k,l},E),
\end{align}
where we have used the fact that $k>i$, since $(k,l)\in I_1$. Using \eqref{eq:bi-Lip 1}, the $\beta/\eps$-codimensional relationship between $\mu$ and $\nu$ given by Theorem~\ref{thm:HypFill}, as well as \eqref{eq:d(z,E)}, we then obtain 
    \begin{align*}
    \sum_{(k,l)\in I_1}|(Tu)_{U_{k,l}}|^p\int_{B_{k,l}}\frac{d\mu(x)}{d_\eps(x,E)^p}&\le\sum_{(i,j)\in I_E}\sum_{(k,l)\in I_{i,j}}|(Tu)_{U_{k,l}}|^p\int_{B_{k,l}}\frac{d\mu(x)}{d_\eps(x,E)^p}\\
    &\lesssim\sum_{(i,j)\in I_E}\sum_{(k,l)\in I_{i,j}}\frac{\mu(B_{k,l})}{\nu(U_{k,l})}\int_{U_{k,l}}\frac{|Tu(z)|^p}{d(U_{k,l},E)^p}d\nu(z)\\
    &\lesssim\sum_{(i,j)\in I_E}\sum_{(k,l)\in I_{i,j}}\alpha^{-k\beta/\eps}\int_{U_{k,l}}\frac{|Tu(z)|^p}{d(U_{k,l},E)^p}d\nu(z)\\
    &\lesssim\sum_{(i,j)\in I_E}\sum_{(k,l)\in I_{i,j}}\alpha^{-k\beta/\eps}\int_{U_{k,l}}\frac{|Tu(z)|^p}{d(z,E)^p}d\nu(z).
    \end{align*}

We note that for each $k$, the collection $\{U_{k,l}\}_{l\in I_k}$ has bounded overlap since $\nu$ is doubling and $\{z_{k,l}\}_{l\in I_k}$ is $\alpha^{-k}$-separated.  Furthermore, by our choice of $\alpha$, if $(k,l)\in I_{i,j}$, then $U_{k,l}\subset 3U_{i,j}$. From Lemma~\ref{lem:InducedCover}, it also follows that if $(i,j)\in I_E$ and $z\in 3U_{i,j}$, then $d(z,E)\ge 4\alpha^{-i}$. From these facts, as well as the bounded overlap of $\{3U_{i,j}\}_{(i,j)\in I_E}$, also given by Lemma~\ref{lem:InducedCover}, we have that
    \begin{align*}
    \sum_{(i,j)\in I_E}\sum_{(k,l)\in I_{i,j}}\alpha^{-k\beta/\eps}\int_{U_{k,l}}\frac{|Tu(z)|^p}{d(z,E)^p}d\nu(z)&=\sum_{(i,j)\in I_E}\sum_{k=i+1}^\infty\sum_{\substack{l\,s.t.\,\\(k,l)\in I_{i,j}}}\alpha^{-k\beta/\eps}\int_{U_{k,l}}\frac{|Tu(z)|^p}{d(z,E)^p}d\nu(z)\\
    &\lesssim\sum_{(i,j)\in I_E}\sum_{k=i+1}^\infty\alpha^{-k\beta/\eps}\int_{3U_{i,j}}\frac{|Tu(z)|^p}{d(z,E)^p}d\nu(z)\\
    &\simeq\sum_{(i,j)\in I_E}\alpha^{-i\beta/\eps}\int_{3U_{i,j}}\frac{|Tu(z)|^p}{d(z,E)^p}d\nu(z)\\
    &\lesssim\sum_{(i,j)\in I_E}\int_{3U_{i,j}}\frac{|Tu(z)|^p}{d(z,E)^{\theta p}}d\nu(z)\lesssim\int_{Z\setminus E}\frac{|Tu(z)|^p}{d(z,E)^{\theta p}}d\nu(z).
    \end{align*}
Note that we have used the assumption that $\beta/\eps=p(1-\theta)$ to obtain the second to last inequality.  Since $u\in \Lip_c(\oXeps\setminus E)$, it follows that $Tu\in \Lip_c(Z\setminus E)$.  Since $Z\setminus E$ satisfies a $(\theta,p)$-Hardy inequality, and by boundedness of the trace operator given by Theorem~\ref{thm:HypFill}, we therefore obtain
    \begin{align}\label{eq:512-4}
    \sum_{(k,l)\in I_1}|(Tu)_{U_{k,l}}|^p\int_{B_{k,l}}\frac{d\mu(x)}{d(x,E)^p}\lesssim\int_Z\int_Z\frac{|Tu(z)-Tu(w)|^p}{d(z,w)^{\theta p}\nu(B(z,d(z,w)))}d\nu(w)d\nu(z)\lesssim\int_{\oXeps}g_u^pd\mu.
    \end{align}

For $(k,l)\in I_2$, we have that $d_\eps(B_{k,l}, E)\ge d_\eps(B_{k,l},Z)\ge 3\alpha^{-k}$ by \eqref{eq:VertexToBoundary}, and so it follows that  
    \begin{align*}
    \sum_{(k,l)\in I_2}|(Tu)_{U_{k,l}}|^p\int_{B_{k,l}}\frac{d\mu(x)}{d_\eps(x,E)^p}&\lesssim\sum_{(k,l)\in I_2}\frac{\mu(B_{k,l})}{\nu(U_{k,l})}\int_{U_{k,l}}\frac{|Tu(z)|^p}{\alpha^{-kp}}d\nu(z)\\
    &\lesssim\sum_{(k,l)\in I_2}\int_{U_{k,l}}\frac{|Tu(z)|^p}{\alpha^{-k\theta p}}d\nu(z)=\sum_{(k,l)\in I_2}\int_{Z\setminus E}\frac{|Tu(z)|^p\chi_{U_{k,l}}(z)}{\alpha^{-k\theta p}}d\nu(z).
    \end{align*}
Here we have again used the $\beta/\eps$-codimensional relationship between $\mu$ and $\nu$ as well as the fact that $\beta/\eps=p(1-\theta)$.  

For $(k,l)\in I_2$, we have that $(k,l)\not\in I_E$, and so by the definition of $I_E$ given in the proof of Lemma~\ref{lem:InducedCover}, it follows that $U_{k,l}\cap N_k=\varnothing$, where $N_k$ is given by \eqref{eq:Neighborhood}.  Hence, either $d(U_{k,l},E)<8\alpha^{-k}$ or $d(U_{k,l},E)\ge 8\alpha^{-(k-1)}$.  If $d(U_{k,l}, E)\ge 8\alpha^{-(k-1)}$, then there exists $(i,j)\in I_E$ with $i>k$ such that $U_{k,l}\cap U_{i,j}\ne\varnothing$, and so $(k,l)\in I_1$.  However, this is a contradiction by the definition of $I_2$, and so it follows that $d(U_{k,l},E)<8\alpha^{-k}$.  Thus, for all $z\in U_{k,l}$, we have that
\begin{equation*}\label{eq:I2 distance}
    d(z,E)<10\alpha^{-k},
\end{equation*}
from which it follows that
\[
k\le -\log_{\alpha}(d(z,E)/10)=:m(z).
\]
Therefore, by Tonelli's theorem and bounded overlap of $\{U_{k,l}\}_{l\in I_k}$, we obtain
\begin{align*}
    \sum_{(k,l)\in I_2}\int_{Z\setminus E}\frac{|Tu(z)|^p\chi_{U_{k,l}}(z)}{\alpha^{-k\theta p}}&d\nu(z)=\int_{Z\setminus E}|Tu(z)|^p\sum_{\substack{(k,l)\in I_2\\\st z\in U_{k,l}}}\frac{1}{\alpha^{-k\theta p}}d\nu(z)\\
    &\le\int_{Z\setminus E}|Tu(z)|^p\sum_{k=-\infty}^{\lceil m(z)\rceil}\sum_{\substack{l\in I_k\\\st z\in U_{k,l}}}\frac{1}{\alpha^{-k\theta p}}d\nu(z)\\
    &\lesssim\int_{Z\setminus E}|Tu(z)|^p\sum_{k=-\infty}^{\lceil m(z)\rceil}\alpha^{k\theta p}d\nu(z)\\
    &\simeq\int_{Z\setminus E}|Tu(z)|^p\alpha^{m(z)\theta p}d\nu(z)\simeq\int_{Z\setminus E}\frac{|Tu(z)|^p}{d(z,E)^{\theta p}}d\nu(z)\lesssim\int_{\oXeps}g_u^pd\mu,
\end{align*}
where the last inequality follows as $Z\setminus E$ satisfies a $(\theta,p)$-Hardy inequality and by boundedness of the trace operator, given by Theorem~\ref{thm:HypFill}.  Combining this with \eqref{eq:512-4}, we obtain the following estimate for the first term of \eqref{eq:512-3}:
\begin{equation}\label{eq:512-5}
    \sum_{(i,j)\in I}|(Tu)_{U_{i,j}}|^p\int_{B_{i,j}}\frac{d\mu(x)}{d_\eps(x,E)^p}\lesssim\int_{\oXeps}g_u^pd\mu.
\end{equation}
Here, the comparison constant depends only on $p$, $\theta$, $\alpha$, $\tau$, $\diam(Z)$, $C_\nu$, and $C_{\theta,p}$. 

We now estimate the second term of \eqref{eq:512-3}.  We have that 
\begin{align}\label{eq:512-6}
\sum_{(i,j)\in I}\int_{B_{i,j}}&\frac{|u(x)-(Tu)_{U_{i,j}}|^p}{d_\eps(x,E)^p}d\mu(x)\nonumber\\
&\lesssim\sum_{(i,j)\in I}\int_{B_{i,j}}\frac{|u(x)-u_{B_{i,j}}|^p}{d_\eps(x,E)^p}d\mu(x)+\sum_{(i,j)\in I}|u_{B_{i,j}}-(Tu)_{U_{i,j}}|^p\int_{B_{i,j}}\frac{d\mu(x)}{d_\eps(x,E)^p}.    
\end{align}
As $(\oXeps,d_\eps,\mu_\beta)$ is doubling and supports a $(1,1)$-Poincar\'e inequality, it also supports a $(p,p)$-Poincar\'e inequality, see \cite{HaKo}, and also \cite[Theorem~9.1.2]{HKST} for example.  Using this fact along with \eqref{eq:VertexToBoundary}, we estimate the first term on the right-hand side of \eqref{eq:512-6} by
\begin{align}\label{eq:414}
    \sum_{(i,j)\in I}\int_{B_{i,j}}\frac{|u(x)-u_{B_{i,j}}|^p}{d_\eps(x,E)^p}d\mu(x)&\lesssim\sum_{(i,j)\in I}\frac{\mu(B_{i,j})}{\alpha^{-ip}}\fint_{B_{i,j}}|u-u_{B_{i,j}}|^pd\mu\lesssim\sum_{(i,j)\in I}\int_{B_{i,j}}g_{u}^pd\mu\lesssim\int_{\oXeps} g_u^pd\mu.
\end{align}
We note that $(\oXeps,d_\eps)$ is a geodesic space, and so the $(p,p)$-Poincare inequality has scaling factor $\lambda=1$, see \cite{HaKo}.  Thus, we obtain the last inequality by bounded overlap of $\{B_{i,j}\}_{(i,j)\in I}$, see Lemma~\ref{lem:HypFillWhitney}.

To estimate the second term on the right-hand side of \eqref{eq:512-6}, we have from \eqref{eq:VertexToBoundary} that 
\begin{align}\label{eq:512-7}
    \sum_{(i,j)\in I}|u_{B_{i,j}}-(Tu)_{U_{i,j}}|^p\int_{B_{i,j}}\frac{d\mu(x)}{d_\eps(x,E)^p} \lesssim\sum_{(i,j)\in I}\frac{\mu(B_{i,j})}{\alpha^{-ip}}\left(\fint_{U_{i,j}}|Tu(z)-u_{B_{i,j}}|d\nu(z)\right)^p.
\end{align}
For each $z\in U_{i,j}$, consider the chain of balls $\{B_k(z)\}_{k=i}^\infty$ given by Lemma~\ref{lem:Chain}. 
Since $\beta/\eps=p(1-\theta)$, we can choose $q>1$ and $\delta>0$ such that $\beta/\eps<q<p$ and $q(1-\delta)>\beta/\eps$. Using Lemma~\ref{lem:Chain} Claims \eqref{eq:Chain 1}, \eqref{eq:Chain 6}, and \eqref{eq:Chain 3}, and applying the $(1,q)$-Poincar\'e inequality and H\"older's inequality twice, we obtain
\begin{align*}
    \fint_{U_{i,j}}|Tu(z)-u_{B_{i,j}}|d\nu(z)&\le\fint_{U_{i,j}}\sum_{k=i}^\infty|u_{B_k(z)}-u_{B_{k+1}(z)}|d\nu(z)\\
    &\lesssim\fint_{U_{i,j}}\sum_{k=i}^\infty\alpha^{-k\delta-k(1-\delta)}\left(\fint_{2B_k(z)}g_u^q\, d\mu\right)^{1/q} d\nu(z)\\
    &\lesssim\alpha^{-i\delta}\fint_{U_{i,j}}\left(\sum_{k=i}^\infty\alpha^{-k(1-\delta)q}\fint_{2B_k(z)}g_u^qd\mu\right)^{1/q}d\nu(z)\\
    &\le\alpha^{-i\delta}\left(\fint_{U_{i,j}}\sum_{k=i}^\infty\alpha^{-k(1-\delta)q}\fint_{2B_k(z)}g_u^qd\mu\, d\nu(z)\right)^{1/q}.
\end{align*}
By Lemma~\ref{lem:Chain} Claim \eqref{eq:Chain 2}, there exists a constant $C\ge 1$, depending only on the bi-Lipschitz constant between $d$ and $d_\eps$, such that $2B_k(z)\subset B_\eps(z,C\alpha^{-k})$.  Using this, as well as Lemma~\ref{lem:Chain} Claims \eqref{eq:Chain 4} and \eqref{eq:Chain 5}, we then obtain
\begin{align*}
    \fint_{U_{i,j}}|Tu(z)-&u_{B_{i,j}}|d\nu(z)\\
        &\lesssim\alpha^{-i\delta}\left(\fint_{U_{i,j}}\sum_{k=i}^\infty\frac{\alpha^{-k(1-\delta)q+k\beta/\eps}}{\nu(B_\eps(z,\alpha^{-k}))}\int_{10B_{i,j}}g_u(x)^q\chi_{B_\eps(z,C\alpha^{-k})}(x)\, d\mu(x)\,d\nu(z)\right)^{1/q}.
\end{align*}

By Tonelli's theorem and our choice of $q$ and $\delta$, we then obtain
\begin{align*}
     \fint_{U_{i,j}}|Tu(z)-&u_{B_{i,j}}|d\nu(z)\\
     &\lesssim\frac{\alpha^{-i\delta}}{\nu(U_{i,j})^{1/q}}\left(\int_{CB_{i,j}}g_u(x)^q\sum_{k=i}^\infty\alpha^{-k(1-\delta)q+k\beta/\eps}\int_{U_{i,j}}\frac{\chi_{B_\eps(z,C\alpha^{-k})}(x)}{\nu(B_\eps(z,\alpha^{-k}))}\, d\nu(z)\,d\mu(x)\right)^{1/q}\\
     &\lesssim\frac{\alpha^{-i+i\beta/(\eps q)}}{\nu(U_{i,j})^{1/q}}\left(\int_{CB_{i,j}}g_u^q\,d\mu\right)^{1/q}.
\end{align*}
Substituting this into \eqref{eq:512-7} and using the $\beta/\eps$-codimensionality between $\mu$ and $\nu$, we have that 
\begin{align*}
    \sum_{(i,j)\in I}|u_{B_{i,j}}-(Tu)_{U_{i,j}}|^p\int_{B_{i,j}}\frac{d\mu(x)}{d_\eps(x,E)^p} &\lesssim\sum_{(i,j)\in I}\frac{(\alpha^{i\beta/\eps})^{p/q}\mu(B_{i,j})}{\nu(U_{i,j})^{p/q}}\left(\int_{10B_{i,j}}g_u^q\,d\mu\right)^{p/q}\\
    &\lesssim\sum_{(i,j)\in I}\mu(B_{i,j})^{1-p/q}\left(\int_{10B_{i,j}}g_u^q\,d\mu\right)^{p/q}\\
    &\le\sum_{(i,j)\in I}\mu(B_{i,j})^{1-p/q}\left(\int_{B_{i,j}}Mg_u^q\,d\mu\right)^{p/q}\\
    &\le\sum_{(i,j)\in I}\int_{B_{i,j}}(Mg_u^q)^{p/q}\,d\mu\lesssim\int_{\oXeps}(Mg_u^q)^{p/q}\,d\mu.
\end{align*}
Here $M$ is the uncentered Hardy-Littlewood maximal function, and we have used H\"older's inequality and the bounded overlap of the $\{B_{i,j}\}_{(i,j)\in I}$ to obtain the last two inequalities, see Lemma~\ref{lem:HypFillWhitney}.  By boundedness of the maximal function from $L^{p/q}(\oXeps,\mu)$ to $L^{p/q}(\oXeps,\mu)$, see \cite[Theorem~3.5.6]{HKST} for example, we obtain the following estimate of the second term in \eqref{eq:512-6}:
    \begin{align*}
    \sum_{(i,j)\in I}|u_{B_{i,j}}-(Tu)_{U_{i,j}}|^p\int_{B_{i,j}}\frac{d\mu(x)}{d_\eps(x,E)^p}\lesssim\int_{\oXeps} g_u^p\,d\mu.
    \end{align*}
Combining this with \eqref{eq:414} and  \eqref{eq:512-6}, the second term of \eqref{eq:512-3} is now estimated by 
\[
\sum_{(i,j)\in I}\int_{B_{i,j}}\frac{|u(x)-(Tu)_{U_{i,j}}|^p}{d_\eps(x,E)^p}d\mu(x)\lesssim\int_{\oXeps}g_u^p\,d\mu,
\]
and so, with \eqref{eq:512-5}, we obtain the desired inequality:
\[
\int_{\oXeps\setminus E}\frac{|u(x)|^p}{d_\eps(x,E)^p}d\mu(x)\lesssim\int_{\oXeps} g_u^p\,d\mu.
\]
Here, the comparison constant depends only on $\theta$, $p$, $\alpha$, $\tau$, $C_\nu$, and $C_{\theta,p}$.
\end{proof}

We now show that for an appropriate choice of $\theta$, $p$, and $\beta$, a $p$-Hardy inequality with respect to $\mu_\beta$ on $\oXeps$ implies a $(\theta,p)$-Hardy inequality on $Z$.

\begin{prop}\label{prop:XtoZ}
    Let $E\subset Z$ be a closed set, let $1\le p<\infty$, $0<\theta<1$, and $\beta>0$ be such that $\beta/\eps=p(1-\theta)$, and suppose that $\oXeps\setminus E$ satisfies a $p$-Hardy inequality with respect to $\mu_\beta$.  Then $Z\setminus E$ satisfies a $(\theta,p)$-Hardy inequality, with constant depending on $\theta$, $p$, $C_\nu$, and $C_p$.
\end{prop}

\begin{proof}
    Let $u\in\Lip_c(Z\setminus E)$, and let $I_E\subset I$ be the indexing set given by Lemma~\ref{lem:InducedCover}.  For ease of notation, we again set $\mu:=\mu_\beta$.  We then have that 
    \begin{align}\label{eq:512-1}
        \int_{Z\setminus E}\frac{|u(z)|^p}{d(z,E)^{\theta p}}d\nu(z)&\le\sum_{(i,j)\in I_E}\int_{U_{i,j}}\frac{|u(z)|^p}{d(z,E)^{\theta p}}d\nu(z)\nonumber\\
        &\lesssim\sum_{(i,j)\in I_E}|(Eu)_{B_{i,j}}|^p\int_{U_{i,j}}\frac{d\nu(z)}{d(z,E)^{\theta p}}+\sum_{(i,j)\in I_E}\int_{U_{i,j}}\frac{|u(z)-(Eu)_{B_{i,j}}|^p}{d(z,E)^{\theta p}}d\nu(z).
    \end{align}
    Here $Eu\in N^{1,p}(\oXeps,\mu_\beta)$ is the extension of $u$ given by Theorem~\ref{thm:HypFill}, see \eqref{eq:HypFillExtension}.    
    
    By Lemma~\ref{lem:InducedCover} and \eqref{eq:VertexToBoundary}, we have that
    \[
    d(U_{i,j},E)\simeq\alpha^{-i}\simeq d_\eps(B_{i,j},U_{i,j}).
    \]
    Using this, along with the $\beta/\eps$-codimensionality between $\mu$ and $\nu$, the assumption that $\beta/\eps=p(1-\theta)$, and bounded overlap of $\{B_{i,j}\}_{(i,j)\in I}$, we estimate the first term on the right-hand side of \eqref{eq:512-1} as follows:
    \begin{align*}
        \sum_{(i,j)\in I_E}|(Eu)_{B_{i,j}}|^p\int_{U_{i,j}}\frac{d\nu(z)}{d(z,E)^{\theta p}}&\lesssim\sum_{(i,j)\in I_E}\frac{\nu(U_{i,j})}{\alpha^{-i\theta p}}\fint_{B_{i,j}}|Eu|^pd\mu\\
        &\simeq\sum_{(i,j)\in I_E}\frac{\nu(U_{i,j})}{\mu(B_{i,j})}\int_{B_{i,j}}\frac{|Eu(x)|^p}{d_\eps(x,E)^{\theta p}}d\mu(x)\\
        &\simeq\sum_{(i,j)\in I_E}\frac{1}{\alpha^{-i\beta/\eps}}\int_{B_{i,j}}\frac{|Eu(x)|^p}{d_\eps(x,E)^{\theta p}}d\mu(x)\\
        &\simeq\sum_{(i,j)\in I_E}\int_{B_{i,j}}\frac{|Eu(x)|^p}{d_\eps(x,E)^{p}}d\mu(x)\lesssim\int_{\oXeps\setminus E}\frac{|Eu(x)|^p}{d_\eps(x,E)^{p}}d\mu(x).
    \end{align*}
    As $u\in \Lip_c(Z\setminus E)$, it follows that $Eu\in\Lip_c(\oXeps\setminus E)$, see \eqref{eq:HypFillExtension}. 
 Since $\oXeps\setminus E$ satisfies a $p$-Hardy inequality with respect to $\mu_\beta$, and by the boundedness of the extension operator, see Theorem~\ref{thm:HypFill}, we have that
    \begin{align}\label{eq:512-2}
        \sum_{(i,j)\in I_E}|(Eu)_{B_{i,j}}|^p\int_{U_{i,j}}\frac{d\nu(z)}{d(z,E)^{\theta p}}&\lesssim\int_{\oXeps} g_{Eu}^pd\mu\nonumber\\
        &\lesssim\int_Z\int_Z\frac{|u(z)-u(w)|^p}{d(z,w)^{\theta p}\nu(B(z,d(z,w)))}d\nu(w)d\nu(z).
    \end{align}

    We now estimate the second term on the right-hand side of \eqref{eq:512-1}.  If $x\in B_{i,j}$, for $(i,j)\in I_E$, then there exist vertices $v_1=(y_1,i_1)$ and $v_2=(y_2,i_2)$ such that $x\in[v_1,v_2]$, and $i_1\simeq i\simeq i_2$.  Thus, by the definition of $E$ given by \eqref{eq:HypFillExtension}, for $z\in U_{i,j}$ and $x\in U_{i,j}$, there exists $k\in\{1,2\}$ such that 
    \[
    |u(z)-Eu(x)|\le \left |u(z)-\fint_{B_Z(y_k,\alpha^{i_k})}u\,d\nu\right|\le\fint_{B_Z(y_k,\alpha^{i_k})}|u(z)-u(w)|d\nu(w).
    \]
    Since $x\in B_{i,j}$ and since $i_k\simeq i$, there exists $C:=C(\alpha,\tau,C_\nu)$ such that 
    \[
    |u(z)-Eu(x)|\lesssim\fint_{CU_{i,j}}|u(z)-u(w)|d\nu(w),
    \]
    where $CU_{i,j}=B_Z(z_{i,j},C\alpha^{-i})\cap Z$, and we have used the doubling property of $\nu$.  Thus, by H\"older's inequality, we have that 
    \begin{align*}
        \int_{U_{i,j}}\frac{|u(z)-(Eu)_{B_{i,j}}|^p}{d(z,E)^{\theta p}}d\nu(z)&\le\frac{1}{\alpha^{-i\theta p}}\int_{U_{i,j}}\fint_{B_{i,j}}|u(z)-Eu(x)|^pd\mu(x)d\nu(z)\\
        &\lesssim\frac{1}{\alpha^{-i\theta p}}\int_{U_{i,j}}\fint_{B_{i,j}}\fint_{CU_{i,j}}|u(z)-u(w)|^pd\nu(w)\,d\mu(x)\,d\nu(z)\\
        &=\frac{1}{\alpha^{-i\theta p}}\int_{U_{i,j}}\fint_{CU_{i,j}}|u(z)-u(w)|^pd\nu(w)\,d\nu(z)\\
        &\lesssim\int_{U_{i,j}}\int_{CU_{i,j}}\frac{|u(z)-u(w)|^p}{d(z,w)^{\theta p}\nu(B(z,d(z,w)))}d\nu(w)d\nu(z)\\
        &\le\int_{U_{i,j}}\int_{Z}\frac{|u(z)-u(w)|^p}{d(z,w)^{\theta p}\nu(B(z,d(z,w)))}d\nu(w)d\nu(z).
    \end{align*}
    By bounded overlap of the collection $\{U_{i,j}\}_{(i,j)\in I_E}$, we then obtain
    \[
    \sum_{(i,j)\in I_E}\int_{U_{i,j}}\frac{|u(z)-(Eu)_{B_{i,j}}|^p}{d(z,E)^{\theta p}}d\nu(z)\lesssim\int_Z\int_Z\frac{|u(z)-u(w)|^p}{d(z,w)^{\theta p}\nu(B(z,d(z,w)))}d\nu(w)d\nu(z).
    \]
    Combining this estimate with \eqref{eq:512-1} and \eqref{eq:512-2} completes the proof.  Here, the comparison constants depend only on $\theta$, $p$, $\alpha$, $\tau$, $C_\nu$, and $C_p$.
\end{proof}

\section{Localization and proof of Theorem~\ref{thm:FracHardyImprovement}}\label{sec:Self-Improvement}

In this section, we combined the results of the previous sections with a localization argument in order to prove Theorem~\ref{thm:FracHardyImprovement}.

\subsection{Localization}\label{sec:localization}
We now let $(Z,d,\nu)$ be a complete doubling metric measure space. In order to apply Theorem~\ref{thm:HypFill Hardy}, which is proven for a \emph{compact} doubling metric measure space, in the proof of Theorem \ref{thm:FracHardyImprovement}, we first need to show that if $E\subset Z$ is a closed set, with $Z\setminus E$ bounded satisfying a $(\theta,p)$-Hardy inequality, then we can find a compact set $Z_0\subset Z$ such that the $\nu|_{Z_0}$ is doubling, $Z_0\setminus E=Z\setminus E$, and that $Z_0\setminus E$ satisfies a $(\theta,p)$-Hardy inequality.  The basic step is finding an exhaustion of $Z$ by sets $Z_R$, $R>0$, so that $Z_R$ equipped with the restricted measures are doubling. This is a much simplified version of an exhaustion by uniform domains constructed by Rajala in \cite{R}. The restricted measure is given by $\nu|_{Z_R}(A)=\nu(Z_R\cap A)$ and the restricted distance is $d|_{Z_R\times Z_R}$. For simplicity, we will use $d$ to denote the distance in both the subset and entire space.
\begin{prop}\label{prop:exhaustion} Let $(Z,d,\nu)$ be complete and doubling and $z_0\in Z$. For every $R>0$, there exists a compact set $Z_R\subset Z$ for which
\begin{enumerate}
    \item $B(z_0,R)\subset Z_R \subset B(z_0,2R)$, and
    \item $(Z_R,d|_{Z_R\times Z_R}, \nu|_{Z_R})$ is doubling.
\end{enumerate}
\end{prop}
\begin{proof}
    Define recursively the following sets
    \[
    \Omega_0 = B(z_0,R),\quad \Omega_{i+1}=\{y\in Z : d(y,\Omega_{i})\leq 4^{-i-1} R).
    \]
    Let 
    \[
Z_R=\overline{\bigcup_{i=0}^\infty \Omega_i}.
    \]
    The set $Z_R$ is closed by definition, and to show that it is compact it suffices to prove that it is bounded. By induction on $i\in \N$ one shows that 
    \[\Omega_i \subset B(z_0, \sum_{k=0}^i R4^{-k}) \subset B(z_0, 4/3 R).\]
    Thus, $Z_R \subset B(z_0, 2R)$ and $Z_R$ is compact. It is also direct that $B(z_0,R)=\Omega_0\subset Z_R$.

    Next, we prove that $Z_R$ equipped with the restricted measure $\nu|_{Z_R}$ is doubling. Let $z\in Z_R$ and let $r>0$. First, if $r>4R$, then 
    \[
    \nu|_{Z_R}(B(z,r))=\nu(B(z,2r)\cap Z_R)=\nu(Z_R)\geq \nu|_{Z_R}(B(z,2r)).
    \]
    Since $\nu(Z_R)\geq \nu(B(z_0,R))>0$, we get doubling for all radii $r>4R$.

    Thus, we are left to consider $r<4R$. For this, we prove a version of the \emph{corkscrew condition} from \cite{BS}. Let $c=1/48$. The corkscrew condition is: For every $z\in Z_R$ there exits an $z_r\in Z_R\cap B(z,r/3)$ so that $B(z_r,cr)\subset B(z,r)\cap Z_R$.

    First, choose $k\in \N$ so that $R4^{-k}<r\leq R4^{1-k}$. By induction, for every $p\in \Omega_i$ and all $i\in \N$ we have $d(p,\Omega_k)\leq 4^{-k}R/3$. This follows directly for $i\leq k$, and induction step follows from the definition. Thus, $d(z,\Omega_k)\leq 4^{-k}R/3$. 

    Pick $z_r\in \Omega_k$ so that $d(z_r, z) \leq 4^{-k}R/3\leq r/3$. Then, 
    \[B(z_r, cr)\subset B(z_r, R4^{-1-k})\subset Z_R\]
    by construction, and $B(z_r,cr)\subset B(z,r)$ by the choice of $c$.
    
    With the corkscrew condition established, the doubling condition is fairly direct. First, by construction:
    \[
    \nu|_{Z_R}(B(z, 2r)) \leq \nu(B(z,2r)),
    \]
    and
    \[
    \nu|_{Z_R}(B(z, r))=\nu(Z_R\cap B(z,r))\geq \nu(B(z_r, cr))\geq C\nu(B(z,2r)),
    \]
    where $C$ is some constant from the doubling bound. It is also direct that $\nu|_{Z_R}(B(z, r))>0$, and thus doubling follows.
\end{proof}

\begin{remark} The proposition above can often be much simplified.
    If $Z$ is geodesic, then it suffices to define $Z_R=\overline{B(z_0,R)}$. Thus, exhaustions by balls work for example in all Euclidean spaces.
\end{remark}

Given this exhaustion, we can prove the following localization argument.

\begin{prop}\label{prop:localization} Let $(Z,d,\nu)$ be complete and doubling, let $0<\theta<1$ and $1<p<\infty$, and let $E\subset Z$ be a closed set such that $Z\setminus E$ is bounded.  Let $z_0\in Z\setminus E$, and let $Z_R$ be the exhaustion constructed in Proposition \ref{prop:exhaustion}.
\begin{itemize}
\item[(i)]If $Z\setminus E$ satisfies a $(\theta,p)$-Hardy inequality, then there exists an $R>0$ so that $Z\setminus E = Z_R\setminus E$ and $Z_R\setminus E$ satisfies a $(\theta,p)$-Hardy inequality.

\item[(ii)]Conversely, if $R>0$ is such that $Z\setminus E = Z_R\setminus E$ and $Z_R\setminus E$ satisfies a $(\theta,p)$-Hardy inequality, then $Z\setminus E$ satisfies a $(\theta,p)$-Hardy inequality.
\end{itemize}
\end{prop}
\begin{proof}
    Consider the inequality for the open set $\Omega=Z\setminus E$ in the space $X=Z$:
    \begin{equation}\label{eq:frachardyproof}
    \int_{Z\setminus E} \frac{|u(z)|^p}{d(z,Z\setminus E)^{\theta p}} d\nu(z) \leq C_{\theta,p} \int_Z\int_Z \frac{|u(z)-u(w)|^p}{d(z,w)^{\theta p} \nu(B(z,d(z,w)))} d\nu(w) d\nu(z)
    \end{equation}
    for all Lipschitz functions $u$ with compact support in $Z\setminus E$. If $R>0$ is such that $Z\setminus E=Z_R\setminus E$, 
    then replacing $Z$ by $Z_R$ in \eqref{eq:frachardyproof} inequality does not alter the left hand side, but potentially decreases the right hand side. This shows immediately that the fractional Hardy inequality for $\Omega=Z_R\setminus E$ and $X=Z_R$ implies the same for $\Omega=Z\setminus E$ and $X=Z$.  This gives us (ii).

    We now prove (i). Suppose that $Z\setminus E$ satisfies a $(\theta,p)$-Hardy inequality, and let $R>2\diam(Z\setminus E)$ be such that $Z\setminus E=Z_R\setminus E$. We will show that the conclusion holds for $R$ sufficiently large.  To this end, we estimate the right hand side of \eqref{eq:frachardyproof}
    as follows:
    
\begin{align*}
    \int_Z\int_Z &\frac{|u(z)-u(w)|^p}{d(z,w)^{\theta p} \nu(B(z,d(z,w)))} d\nu(w)d\nu(z)\\
    &=\int_{Z_R}\int_{Z_R} \left(\cdots\right)d\nu(w) d\nu(z)+\int_{Z_R}\int_{Z\setminus Z_R}(\cdots)d\nu(w)d\nu(z)+\int_{Z\setminus Z_R}\int_{Z_R}(\cdots)d\nu(w)d\nu(z)  \\
    % &\int_{Z\setminus B(z_0,R)}\int_Z \frac{|u(x)-u(y)|^p}{d(x,y)^{\theta p} \mu(B(x,d(x,y)))} d\mu d\mu + \int_Z\int_{Z\setminus B(z_0,R)}\frac{|u(x)-u(y)|^p}{d(x,y)^{\theta p} \mu(B(x,d(x,y)))} d\mu d\mu\\
    &\leq \int_{Z_R}\int_{Z_R} \frac{|u(z)-u(w)|^p}{d(z,w)^{\theta p} \nu(B(z,d(z,w)))} d\nu(w) d\nu(z) + \int_{E^c}\int_{ B(z_0,R)^c}\frac{|u(z)|^pd\nu(w) d\nu(z)}{d(z,w)^{\theta p} \mu(B(z,d(z,w)))} \\
    &\hspace{+3cm}+
\int_{B(z_0,R)^c}\int_{E^c}\frac{|u(w)|^p}{d(z,w)^{\theta p} \mu(B(z,d(z,w)))} d\nu(w) d\nu(z). \\
&\leq \int_{Z_R}\int_{Z_R} \frac{|u(z)-u(w)|^p}{d(z,w)^{\theta p} \nu(B(z,d(z,w)))} d\nu(w) d\nu(z) \\
    &\hspace{+3cm}+ (1+C_{\nu})\int_{E^c}\int_{ B(z_0,R)^c}\frac{|u(z)|^p}{d(z,w)^{\theta p} \nu(B(z,d(z,w)))} d\nu(w) d\nu(z). \\
\end{align*}
% The goal next will be to absorb the second term on the right to the left-hand side. 

Writing $B(z_0,R)^c=\bigcup_{k=0}^\infty B(z_0,2^{k+1} R)\setminus B(z_0,2^k R)$, we see that if $z\in E^c$ and $w\in B(z_0,2^{k+1} R)\setminus B(z_0,2^k R)$ then $d(z,w)\geq d(z_0,w)-d(z_0,z)\geq 2^{k}R-R/2=2^{k-1}R$. We also have that $B(z,d(z,w))\subset B(z_0,d(z,w)+d(z_0,z))\subset B(z_0,2^{k+2}R)$. Using this and doubling, we get for $z\in E^c$,
\begin{align*}
\int_{ B(z_0,R)^c} \frac{1}{d(z,w)^{\theta p} \nu(B(z,d(z,w)))} d\nu(w) &\leq \sum_{k=0}^\infty \int_{ B(z_0,2^{k+1} R)\setminus B(z_0,2^k R)} \frac{1}{d(z,w)^{\theta p} \nu(B(z,d(z,w)))} d\nu(w) \\
&\leq R^{-\theta p}\sum_{k=0}^\infty\frac{\nu(B(z_0,2^{k+1}R)\setminus B(z_0,2^kR))}{2^{(k-1)\theta p}\nu(B(z,2^{k-1}R))} \\
&\leq C_1R^{-\theta p},
\end{align*}
where $C_1\ge 1$ is a constant depending only on $C_\nu$, $\theta$, and $p$. 
%Similarly
%    \begin{align*}
%\int_{Z\setminus B(z_0,R)} \frac{1}{d(x,y)^{\theta p} \mu(B(x,d(x,y)))} d\mu(y) &\leq \sum_{k=0}^\infty \int_{ B(z_0,2^{k+1} R)\setminus B(z_0,2^k R)} \frac{1}{d(x,y)^{\theta p} \mu(B(x,d(x,y)))} d\mu \\
%&\leq C_\mu^2 \sum_{k=0}^\infty \frac{1}{2^{(k-1)\theta p}R^{\theta p}}  \\
%&\leq C_\mu^2 \sum_{k=0}^\infty \frac{1}{2^{(k-1)\theta p}R^{\theta p}}  \\
%&\leq \frac{4C_\mu^2}{16C_{\mu}^2(2+C_{\theta,p}) d_\Omega(x)^{\theta p}} \leq %\frac{1}{4C_{\theta,p}(1+C_\mu)}\frac{1}{d_\Omega(x)^{\theta p}}.
%\end{align*}
Hence, if $R>(2C_{\theta,p}C_1(1+C_\nu))^{\frac{1}{\theta p}}$, then together with the previous bounds, we obtain
\begin{align*}
    \int_{Z\setminus E}& \frac{|u(z)|^p}{d(z,Z\setminus E)^{\theta p}} d\nu(z) \leq C_{\theta,p} \int_Z\int_Z \frac{|u(z)-u(w)|^p}{d(z,w)^{\theta p} \nu(B(z,d(z,w)))} d\nu(w) d\nu(z) \\
    &\leq C_{\theta,p} \int_{Z_R}\int_{Z_R} \frac{|u(z)-u(w)|^p}{d(z,w)^{\theta p} \mu(B(z,d(z,w)))} d\nu(w) d\nu(z) + \frac{1}{2} \int_{Z\setminus E} \frac{|u(z)|^p}{d(z,Z\setminus E)^{\theta p}} d\nu(z).
\end{align*}
Absorbing the second term to the left hand side yields the fractional Hardy inequality for $Z_R\setminus E$ in $Z_R$, which establishes (i).
\end{proof}

\subsection{Proof of Theorem~\ref{thm:FracHardyImprovement}}

The localization argument above reduces the proof of self-improvement in (possibly unbounded) $Z$ to the self-improvement result in the compact and doubling space $Z_R$ for some $R>0$, where we are able to apply Theorem~\ref{thm:HypFill Hardy}.

\begin{proof}[Proof of Theorem~\ref{thm:FracHardyImprovement}] 
By Proposition~\ref{prop:localization}, we may assume without loss of generality that $Z$ is compact.  We also note that the $(\theta,p)$-Hardy inequality is invariant under scaling of the metric.  That is, for all $\lambda>0$, $Z\setminus E$ satisfies a $(\theta,p)$-Hardy inequality with respect to $d$ if and only if it does so with respect to $\lambda d$, with the same constant.  Therefore, replacing $d$ with $d/(2\diam(Z))$ if necessary, we may assume without loss of generality that $\diam(Z)<1$.

Fix $\alpha=e^{1/4}$ $\tau=2$, and let $\eps:=\log\alpha=1/4$.  Choose $\beta_0>0$ satisfying
\[
\frac{\beta_0}{\eps}=p_0(1-\theta_0),
\]
and consider the uniformized hyperbolic filling $(\oXeps,d_\eps,\mu_{\beta_0})$ of $(Z,d,\nu)$ as constructed in Section~\ref{sec:HypFill}.  As $Z\setminus E$ satisfies a $(\theta_0,p_0)$-Hardy inequality, it follows from Theorem~\ref{thm:HypFill Hardy} that $\oXeps\setminus E$ satisfies a $p_0$-Hardy inequality with respect to $\mu_{\beta_0}$, with constant $C_{p_0}:=C_{p_0}(\theta_0,p_0,C_{\theta_0,p_0},C_\nu)$.

Note that by Theorem~\ref{thm:HypFill}, $(\oXeps,d_\eps,\mu_{\beta_0})$ is doubling and supports a $(1,1)$-Poincar\'e inequality, with doubling constant $C_{\mu_{\beta_0}}$ and Poincar\'e inequality constants depending only on $\theta_0$, $p_0$, and $C_\nu$. As such, it follows from Theorem~\ref{thm:KoskelaZhong} that there exists \[
\eps_1:=\eps_1(\theta_0,p_0,C_{\theta_0,p_0},C_\nu)<\min\{\theta_0p_0/2,p_0-1\}
\]
and $C_1:=C_1(\theta_0,p_0,C_{\theta_0,p_0},C_\nu)$ such that $\oXeps\setminus E$ satisfies a $p$-Hardy inequality with respect to $\mu_{\beta_0}$, with constant $C_p=C_1$, for all $p\in (p_0-\eps_1,p_0+\eps_1)$. 

For such $p$, let 
\[
\theta_p:=1-\frac{p_0(1-\theta_0)}{p}.
\]
Note that our choice of $\eps_1<\min\{\theta_0p_0/2,p_0-1\}$ ensures that $p>1$ and $0<\theta_p<1$.  Since
\[
\frac{\beta_0}{\eps}=p(1-\theta_p),
\]
it follows from Theorem~\ref{thm:HypFill Hardy} that $Z\setminus E$ satisfies a $(\theta_p,p)$-Hardy inequality for all $p\in (p_0-\eps_1,p_0+\eps_1)$, with constant $C_{\theta_p,p}$ depending only on $p$, $\theta_0$, $p_0$, $C_{\theta_0,p_0}$, and $C_\nu$, see Figure~\ref{fig:KoskelaZhongImprovement}.
\begin{figure}[h]
\centering
\includegraphics[scale=0.5]{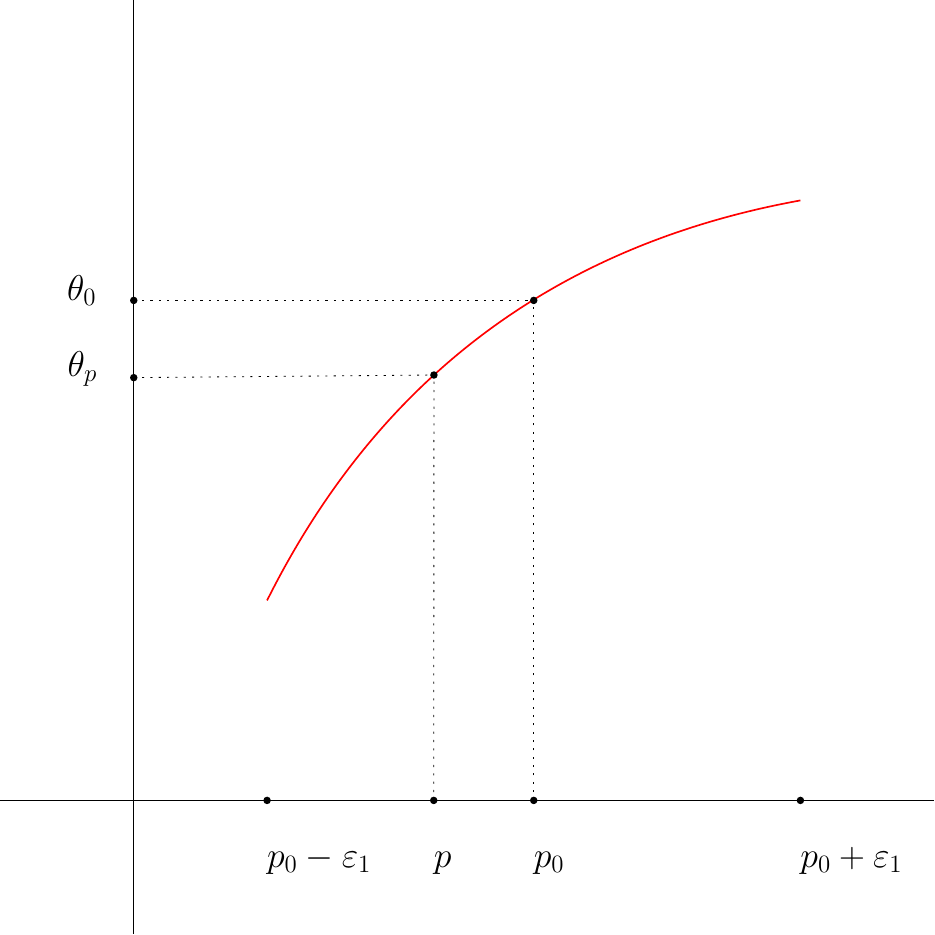}
\caption{For each point $(\theta,p)$ lying on the above curve in the $(\theta,p)$-plane, $Z\setminus E$ satisfies a $(\theta,p)$-Hardy inequality.}\label{fig:KoskelaZhongImprovement}
\end{figure}

Now, let $p\in (p_0-\eps_1,p_0+\eps_1)$. Since $\oXeps\setminus E$ satisfies a $p$-Hardy inequality with respect to $\mu_{\beta_0}$, it follows from Theorem~\ref{thm:delta reg improvement} that $\oXeps\setminus E$ also supports a $p$-Hardy inequality with respect to $w\,d\mu_{\beta_0}$ for any $p$-admissible weight $w$ which is $\delta_p$-regularizable in $\oXeps\setminus E$, with 
\begin{equation}\label{eq:delta p}
\delta_p:=\frac{p}{C_0}(2C_1)^{-1/p}.
\end{equation}
Here, $C_0$ is the constant from Lemma~\ref{lem:regularized gradient}, which depends only on $C_{\mu_{\beta_0}}$, which itself depends only on $\theta_0$, $p_0$, and $C_\nu$.  As $Z$ is a $c$-porous subset of $\oXeps$ by Lemma~\ref{lem:Z porous in X}, with $c$ depending only on $\alpha$, it follows from Proposition~\ref{prop:DistanceToPorousSets} that there exists 
\[
0<\sigma_p:=\sigma_p(\theta_0,p_0,C_\nu,\delta_p)<\min\{\theta_0p_0/2,p_0(1-\theta_0)\}
\]  
such that if $\sigma\in\R$ satisfies $|\sigma|<\sigma_p$, then $d_\eps(\cdot,Z)^\sigma$ is $\delta_p$-regularizable in $\oXeps\setminus E$.  As 
\begin{equation}\label{eq:HypFillWeights}
d_\eps(\cdot,Z)^\sigma\,d\mu_{\beta_0}\simeq \mu_{\beta_0+\eps\sigma},
\end{equation}
see \eqref{eq:WeightedMeasure}, it follows from Theorem~\ref{thm:HypFill} that $d_\eps(\cdot,Z)^\sigma$ is a $p$-admissible weight.  Thus, from Theorem~\ref{thm:delta reg improvement} and \eqref{eq:HypFillWeights}, it follows that $\oXeps\setminus E$ satisfies a $p$-Hardy inequality with respect to $\mu_{\beta_0+\eps\sigma}$ for all $\sigma\in\R$ such that $|\sigma|<\sigma_p$. 

For such $\sigma$, we have that
\[
\frac{\beta_0+\eps\sigma}{\eps}=p(1-(\theta_p-\sigma/p)).
\]
Note that our choice of $\sigma_p<\min\{\theta_0p_0/2,p_0(1-\theta_0)\}$ ensures that $0<\theta_p-\sigma/p<1$ for all such $\sigma$.  From Theorem~\ref{thm:HypFill Hardy}, it then follows that $Z\setminus E$ satisfies a $(\theta_p-\sigma/p,p)$-Hardy inequality for all $\sigma\in\R$ with $|\sigma|<\sigma_p$. We note from Proposition~\ref{prop:DistanceToPorousSets} and \eqref{eq:delta p} that $\sigma_p$ is continuous in $p$ on the interval $(p_0-\eps_1,p_0+\eps_1)$.  Since $\theta_p$ is also continuous in $p$ on this interval, there exists $0<\eps_0\le \eps_1$, depending only on $\theta_0$, $p_0$, $C_{\theta_0,p_0}$, and $C_\nu$ such that $Z\setminus E$ satisfies a $(\theta,p)$-Hardy inequality for all $\theta\in (\theta_0-\eps_0,\theta_0+\eps_0)$ and $p\in(p_0-\eps_0,p_0+\eps_0)$, see Figure~\ref{fig:full improvement}.
\begin{figure}[h]
\centering
\includegraphics[scale=0.5]{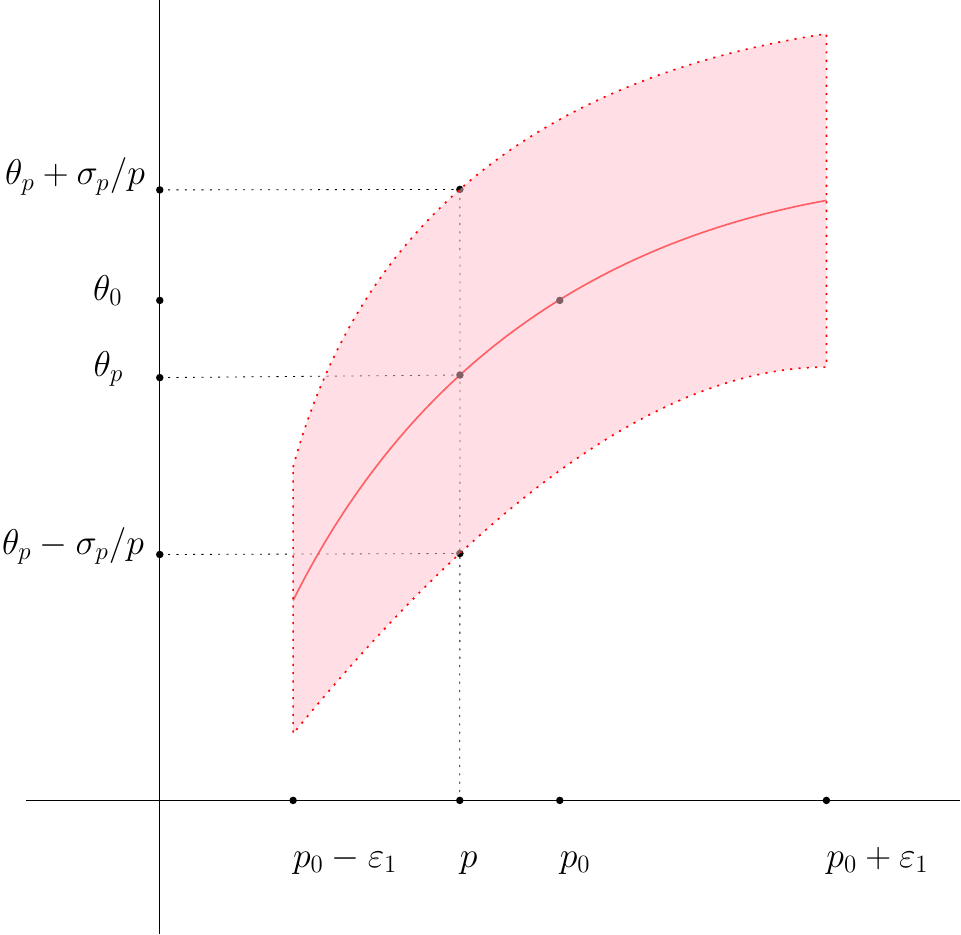}
\caption{For each point $(\theta,p)$ lying in the shaded region of the $(\theta,p)$-plane, $Z\setminus E$ satisfies a $(\theta,p)$-Hardy inequality.}\label{fig:full improvement}
\end{figure}
\end{proof}

\section{Examples and Applications}\label{sec:Examples}

As mentioned in the introduction, self-improvement results for the pointwise $(\theta,p)$-Hardy inequality were recently obtained in \cite{IMV} for domains in doubling metric measure spaces whose complements satisfy the $(\theta,p)$-capacity density condition, see Definition~\ref{def:FracCapDensity} below.  However, our self-improvement result Theorem~\ref{thm:FracHardyImprovement} applies to bounded domains which satisfy the weaker condition of a $(\theta,p)$-Hardy inequality.  In this subsection, we provide an example of such a domain, namely the punctured unit ball $B(0,1)\setminus\{0\}\subset\R^n$, and show that it satisfies a $(\theta,p)$-Hardy inequality whenever $0<\theta<1$ and  $1<p<\infty$ satisfy $\theta p<n$.  Rather than proving this inequality directly, we first show that a sufficient condition for the $(\theta,p)$-Hardy inequality, given in terms of Assouad codimensions, follows from Theorem~\ref{thm:HypFill Hardy} and an analogous result due to Lehrb\"ack \cite{L2} for the $p$-Hardy inequality. In addition to justifying the example of the punctured ball, we provide this sufficient condition, Proposition~\ref{prop:JuhaFractional} below, in an attempt to illustrate how results for fractional Hardy inequalities can be readily obtained from their $p$-Hardy inequality counterparts by using Theorem~\ref{thm:HypFill Hardy}.    

We first recall the definitions of the upper and lower Assouad codimensions.  Given a metric measure space $(X,d,\mu)$, a set $E\subset X$, and $r>0$, we denote the $r$-neighborhood of $E$ by 
\[
E_r:=\{x\in X:d(x,E)<r\}.
\]
The \emph{lower Assouad codimension} of $E$, denoted $\underline{\codim}^\mu_A(E)$, is the supremum of all $t\ge 0$ for which there exists $C\ge 1$ such that 
\[
\frac{\mu(B(x,R)\cap E_r)}{\mu(B(x,R))}\le C\left(\frac{r}{R}\right)^t
\]
for all $x\in E$ and all $0<r\le R<\diam(X)$.  Likewise, the \emph{upper Assouad codimension} of $E$, denoted $\overline{\codim}_A^\mu(E)$, is the infimum of all $s\ge 0$ for which there exists $c>0$ such that 
\[
\frac{\mu(B(x,R)\cap E_r)}{\mu(B(x,R))}\ge c\left(\frac{r}{R}\right)^s
\]
for all $x\in E$ and all $0<r\le R<\diam(E)$. If $\diam(E)=0$, then the upper bound on $R$ is omitted in the definition. When the lower Assouad codimension equals the upper Assouad codimension we simply call it the Assouad codimension.

The following sufficient condition for $p$-Hardy inequalities is due to Lehrb\"ack \cite{L2}:
\begin{prop}\label{prop:Juha Assouad}{\cite[Proposition~7.1]{L2}} Let $(X,d,\mu)$ be a doubling metric measure space supporting a $(1,p)$-Poincar\'e inequality.  Let $\Omega_0\subset X$ be an open set satisfying $\overline{\codim}^\mu_A(X\setminus\Omega_0)<p$.  If $F\subset\overline\Omega_0$ is a closed set satisfying $\underline{\codim}^\mu_A(F)>p$, then $\Omega:=\Omega_0\setminus F$ satisfies a $p$-Hardy inequality with respect to $\mu$.    
\end{prop}

Notice that the metric space $Z$ has Assouad codimension $\beta/\eps=(1-\theta)p$ in its hyperbolic filling $(\oXeps,d_\eps,\mu_\beta)$.  By a similar calculation, a set $A\subset Z$ of Assouad codimension $\theta p$ will have Assouad codimension $\theta p+\beta/\eps=p$ in $\oXeps$. We will make these arguments precise with the aid of covering arguments by balls in the proof of the following statement. 
 Applying this observation and the previous result in conjunction with Theorem~\ref{thm:HypFill Hardy}, we obtain the following fractional analog of the previous proposition.

\begin{prop}\label{prop:JuhaFractional} Let $(Z,d,\nu)$ be a complete doubling metric measure space, and let $0<\theta<1$, and $1<p<\infty$.  Let $\Omega_0\subset Z$ be a bounded open set satisfying $\overline{\codim}^\nu_A(Z\setminus\Omega_0)<\theta p$.  If $F\subset\overline\Omega_0$ is a closed set satisfying $\underline{\codim}^\nu_A(F)>\theta p$, then $\Omega:=\Omega_0\setminus F$ satisfies a $(\theta,p)$-Hardy inequality.    
\end{prop}

\begin{proof}
    Let $z_0\in\Omega_0$.  By Proposition~\ref{prop:exhaustion}, there exists a compact set $Z_0$ with \[
    B(z_0,4\diam(\Omega_0))\subset Z_0\subset B(z_0,8\diam(\Omega_0))
    \]
    such that $(Z_0,d_0,\nu_0)$ is doubling, where $d_0:=d|_{Z_0\times Z_0}$ and  $\nu_0:=\nu|_{Z_0}$.
    
    We first show that $\overline\codim_A^{\nu_0}(Z_0\setminus\Omega_0)<\theta p$.  To this end, let $E:= Z\setminus\Omega_0$, let $z\in E\cap Z_0$, and let $0<r\le R<\diam(E\cap Z_0)$.  Consider the case that $z\in B(z_0,3\diam(\Omega_0))$.  Then since $R<\diam(E\cap Z_0)\le 16\diam(\Omega_0)$, it follows that $B(z,R/16)\subset Z_0$.  If $r\le R/16$, we then have that 
    \begin{equation}\label{eq:Restricted 1}
        \frac{\nu(B(z,R)\cap E_r\cap Z_0)}{\nu(B(z,R)\cap Z_0)}\ge\frac{\nu(B(z,R/16)\cap E_r)}{\nu(B(z,R))}\gtrsim\frac{\nu(B(z,R/16)\cap E_r)}{\nu(B(z,R/16))},
    \end{equation}
    where we have used the doubling property of $\nu$ in the last inequality.
    If $r>R/16$, then $B(z,R/16)\subset E_r$, and so we have that 
    \[
    \nu(B(z,R)\cap E_r\cap Z_0)\ge\nu(B(z,R/16)\cap E_r)=\nu(B(z,R/16))\gtrsim\nu(B(z,R))\ge\nu(B(z,R)\cap E_r).
    \]
     Hence, we have that 
    \begin{equation}\label{eq:Restricted 2}
        \frac{\nu(B(z,R)\cap E_r\cap Z_0)}{\nu(B(z,R)\cap Z_0)}\gtrsim\frac{\nu(B(z,R)\cap E_r)}{\nu(B(z,R))}.
    \end{equation}
    In the case that $z\not\in B(z_0,3\diam(\Omega_0))$, it follows that $B(z,R/16)\cap\Omega_0=\varnothing$, and so $B(z,R/16)\subset E_r$.  In the proof of Proposition~\ref{prop:exhaustion}, it was shown that $Z_0$ satisfies the corkscrew condition with constant $c=1/48$. As such, there exists $z'\in B(z,R/16)$ such that $B(z',cR/16)\subset B(z,R/16)\cap Z_0$, and so it follows that 
    \[
    \nu(B(z,R)\cap E_r\cap Z_0)\ge\nu(B(z',cR/16))\gtrsim\nu(B(z,R))\ge\nu(B(z,R)\cap E_r),
    \]
    where we have used the doubling property of $\nu$.  Thus, we have
    \begin{equation}\label{eq:Restricted 3}
        \frac{\nu(B(z,R)\cap E_r\cap Z_0)}{\nu(B(z,R)\cap Z_0)}\gtrsim\frac{\nu(B(z,R)\cap E_r)}{\nu(B(z,R))}
    \end{equation}
    in this case as well.  Since $\overline\codim_A^\nu(E)<\theta p$, it then follows from \eqref{eq:Restricted 1}, \eqref{eq:Restricted 2}, and \eqref{eq:Restricted 3} that $\overline\codim_A^{\nu_0}(Z_0\setminus\Omega_0)<\theta p$.

    Let $F\subset\Omega_0$ be a closed set such that $\underline\codim_A^\nu(F)>\theta p$.  We then have that $\underline\codim_A^{\nu_0}(F)>\theta p$.  Indeed, let $z\in F$, and let $0<r\le R<\diam (Z_0)$.  As $Z_0$ satisfies the corkscrew condition, there exists $z'\in B(z,R)$ such that $B(z',cR)\subset B(z,R)\cap Z_0$.  Thus, by the doubling property of $\nu$, we have that 
    \[
    \frac{\nu(B(z,R)\cap F_r\cap Z_0)}{\nu(B(z,R)\cap Z_0)}\le\frac{\nu(B(z,R)\cap F_r)}{\nu(B(z',cR))}\lesssim\frac{\nu(B(z,R)\cap F_r)}{\nu(B(z,R))}.
    \]
    As $\underline\codim_A^\nu(F)>\theta p$, it then follows that $\underline\codim_A^{\nu_0}(F)>\theta p$ as well.

    Let $\alpha=e^{1/4}$, $\tau=2$, and let $\eps=\log\alpha$.  Choosing $\beta>0$ such that $\beta/\eps=p(1-\theta)$, consider the uniformized hyperbolic filling $(\oXeps,d_\eps,\mu_\beta)$ of the compact, doubling metric measure space $(Z_0,d_0,\nu_0)$, as constructed in Section~\ref{sec:HypFill}.  Let $E_0:=Z_0\setminus\Omega_0\subset\oXeps$.
    
    We claim that $\overline\codim_A^{\mu_\beta}(E_0)<p$.  Let $x\in E_0$, and let $0<r\le R<\diam_\eps(E_0)$.  We denote the $r$-neighborhood of $E_0$ with respect to the metric $d_\eps$ by $(E_0)_{r,\eps}$.  We note that by Theorem~\ref{thm:HypFill}, there exists a constant $C\ge 2$, depending only on $\alpha$ and $\tau$ such that 
    \begin{align}\label{eq:biLip constant}
        C^{-1}d_0(z,w)\le d_\eps(z,w)\le Cd_0(z,w)
    \end{align}
    for all $z,w\in Z_0$.  Suppose first that $0<r\le R/(2C)$.  In this case, we have that $ B_\eps(x,R/2)\cap (E_0)_{r,\eps}$ is covered by the collection $\{B_\eps(y,r)\}_{y\in B_\eps(x,R/2)\cap E_0}$, and so by the $5$-covering lemma, there exists a disjoint subcollection $\{B_i^\eps:=B_\eps(y_i,r)\}_{i\in I\subset\N}$ such that 
    \[
    (E_0)_{r,\eps}\cap B_\eps(x,R/2)\subset\bigcup_i 5B_i^\eps.
    \]
    Note that by the assumption on $r$, we have that $B_i^\eps\subset (E_0)_{r,\eps}\cap B_\eps(x,R)$ for each $i\in I$.  Thus, by \eqref{eq:biLip constant} and the $\beta/\eps$-codimensionality between $\nu_0$ and $\mu_{\beta}$ given by Theorem~\ref{thm:HypFill}, as well as the doubling property of $\nu_0$, we have that 
    \begin{align*}
        \mu_\beta( B_\eps(x,R)\cap (E_0)_{r,\eps})\ge\sum_{i\in I}\mu_\beta(B_i^\eps)&\simeq r^{\beta/\eps}\sum_{i\in I}\nu_0(B_i^\eps\cap Z)\\
        &\simeq r^{\beta/\eps}\sum_{i\in I}\nu_0(5B_i^\eps\cap Z)\\
        &\ge r^{\beta/\eps}\nu_0(B_\eps(x,R/2)\cap(E_0)_{r,\eps}\cap Z)\\
        &\ge r^{\beta/\eps}\nu_0(B(x,R/(2C))\cap(E_0)_{r/C}),
    \end{align*}
where in the last expression, the $r/C$-neighborhood of $E_0$ and the ball $B(x,R/(2C))$ are both with respect to the metric $d_0$.  By the $\beta/\eps$-codimensionality between $\nu_0$ and $\mu_\beta$ as well as \eqref{eq:biLip constant} and the doubling property of $\nu_0$, it also follows that $\mu_\beta(B_\eps(x,R))\simeq R^{\beta/\eps}\nu_0(B(x,R/(2C))$.  By our choice of $\beta>0$, we then have that 
\begin{equation}\label{eq:Codim 1}
    \frac{\mu_\beta( B_\eps(x,R)\cap (E_0)_{r,\eps})}{\mu_\beta(B_\eps(x,R))}\gtrsim\left(\frac{r}{R}\right)^{p(1-\theta)}\frac{\nu_0( B(x,R/(2C))\cap (E_0)_{r/C})}{\nu_0(B(x,R/(2C)))}.
\end{equation}

Suppose now that $R/(2C)<r\le R$.  In this case, we have that $B_\eps(x,R/(2C))\subset (E_0)_{r,\eps}$, and so it follows from codimensionality, \eqref{eq:biLip constant}, and the doubling property of $\nu_0$ that
\begin{align*}
    \mu_\beta(B_\eps(x,R)\cap(E_0)_{r,\eps})\ge\mu_\beta(B_\eps(x,R/(2C))&\simeq r^{\beta/\eps}\nu_0(B_\eps(x,R/(2C))\cap Z)\\
    & \simeq r^{\beta/\eps}\nu_0(B(x,R/C))\\
    &\ge r^{\beta/\eps}\nu_0(B(x,R/C)\cap (E_0)_{r/C}).
\end{align*}
Hence, by the choice of $\beta>0$, we have in this case that 
\begin{equation*}
    \frac{\mu_\beta( B_\eps(x,R)\cap (E_0)_{r,\eps})}{\mu_\beta(B_\eps(x,R))}\gtrsim\left(\frac{r}{R}\right)^{p(1-\theta)}\frac{\nu_0( B(x,R/C)\cap (E_0)_{r/C})}{\nu_0(B(x,R/C))}.
\end{equation*}
We note that $R/C\le\diam(E_0)$, where the diameter is with respect to $d_0$.  Therefore, since $\overline\codim_A^{\nu_0}(E_0)<\theta p$, it follows from combining this case with \eqref{eq:Codim 1} 
that $\overline\codim_A^{\mu_\beta}(E_0)<p$. 

By a similar argument, using the $\beta/\eps$-codimensionality between $\mu_\beta$ and $\nu_0$, $\eqref{eq:biLip constant}$, the doubling property of $\nu_0$, and the fact that $\underline\codim_A^{\nu_0}(F)>\theta p$, we have that $\underline\codim_A^{\mu_\beta}(F)>p$.  By applying Proposition~\ref{prop:Juha Assouad}, it then follows that $\oXeps\setminus(E_0\cup F)$ satisfies a $p$-Hardy inequality with respect to $\mu_\beta$.  From Theorem~\ref{thm:HypFill Hardy}, we then have that $Z_0\setminus(E_0\cup F)$ satisfies a $(\theta,p)$-Hardy inequality, and so by possibly increasing only the right-hand side of the inequality, it follows that $\Omega:=Z\setminus(E_0\cup F)$ satisfies a $(\theta,p)$-Hardy inequality as well.     \end{proof}

\begin{example}\label{ex:punctured ball}
    Let $\Omega:=B(0,1)\setminus\{0\}\subset\R^n$, and let $0<\theta<1$ and $1<p<\infty$ be such that $\theta p<n$.  Letting $\nu=\Leb^n$, $E:=\R^n\setminus B(0,1)$ and $F:=\{0\}$, we have that 
    \begin{equation*}
        \overline{\codim}_A^\nu(E)=0<\theta p\qquad\text{and}\qquad\underline{\codim}_A^\nu(F)=n>\theta p.
    \end{equation*}
    Thus, from Proposition~\ref{prop:JuhaFractional}, it follows that $\Omega$ satisfies a $(\theta,p)$-Hardy inequality.
    \end{example}

Using the hyperbolic filling, it is also straightforward to show that the complement of the domain $\Omega$ above does not satisfy the $(\theta,p)$-capacity density condition, as considered in \cite{IMV}.  We first recall the definitions of this condition and the fractional relative capacity, stated here for a metric measure space $(Z,d,\nu)$:

\begin{defn}
    Let $0<\theta<1$ and $1\le p<\infty$, and let $\Lambda\ge 2$.  Let $B\subset Z$ be a ball, and let $E\subset\overline B$ be a closed set.  We then write
    \[
    \vcap_{\theta,p}(E,\,2B,\,\Lambda B):=\inf_u\int_{\Lambda B}\int_{\Lambda B}\frac{|u(z)-u(w)|^p}{d(z,w)^{\theta p}\nu(B(z,d(z,w)))}d\nu(w)d\nu(z),
    \]
    where the infimum is taken over all continuous functions $u:Z\to\R$ such that $u\ge 1$ on $E$ and $u=0$ on $Z\setminus 2B$.
\end{defn}

\begin{defn}\label{def:FracCapDensity}
{\cite[Definition~4.1]{IMV}} Let $0<\theta<1$, $1<p<\infty$, and let $E\subset Z$ be a closed set.  We say that $E$ satisfies the \emph{$(\theta,p)$-capacity density condition} if there are constants $c_0>0$ and $\Lambda>2$ such that 
\[
\vcap_{\theta,p}(E\cap\overline{B(x,r)},\,B(x,2r),\, B(x,\Lambda r))\ge c_0\vcap_{\theta,p}(\overline{B(x,r)},\,B(x,2r),\,B(x,\Lambda r))
\]
for all $x\in E$ and $0<r<\diam(E)/8$.
\end{defn}

Notice that $\vcap_{\theta,p}(\overline{B(x,r)},\,B(x,2r),\,B(x,\Lambda r))>0$ whenever $\Lambda > 2$ if $Z$ is geodesic. 

\begin{example}\label{ex:Zero Capacity}
    Let $\Omega=B(0,1)\setminus\{0\}\subset\R^n$, and let $0<\theta<1$ and $1<p<\infty$ be such that $\theta p<n$.  We show that $\R^n\setminus\Omega$ does not satisfy the $(\theta,p)$-capacity density condition, due to the component $\{0\}$.  
    
    Let $\Lambda>2$, and  let $r>0$.  Let $Z=\overline{B(0,\Lambda r)}$, let $d=d_{\Euc}$, and let $\nu=\Leb^n|_{Z}$.  As above, let $\alpha=e^{1/4}$, $\tau=2$, and $\eps=\log\alpha$, and choose $\beta>0$ such that $\beta/\eps=p(1-\theta)$.  We then consider the uniformized hyperbolic filling $(\oXeps,d_\eps,\mu_\beta)$ of the compact, doubling metric measure space $(Z,d,\nu)$.  For each $0<\eta<r$, consider the continuous function $u_\eta:\oXeps\to\R$ given by 
    \[
    u_\eta(x):=\left(1-\frac{d_\eps(x,\{0\})}{\eta}\right)_+.
    \]
    Then $u_\eta$ is $1/\eta$-Lipschitz in $\oXeps$, and if $T$ is the trace operator given by Theorem~\ref{thm:HypFill}, then $Tu_\eta=1$ on $\{0\}$, and $Tu_\eta=0$ in $Z\setminus B(0,2r)$ for sufficiently small $\eta$, depending only the bi-Lipschitz constants between $d$ and $d_\eps$. It then follows from Theorem~\ref{thm:HypFill} and our choice of $\beta>0$ that 
    \begin{align*}
        \vcap_{\theta,p}(\{0\},\,B(0,2r),\,B(0,\Lambda r))\lesssim\int_{\oXeps} g_{u_\eta}^p\,d\mu_\beta\le \eta^{-p}\mu_\beta(B_\eps(0,\eta))\lesssim\eta^{-\theta p}\nu(B(0,\eta))\simeq \eta^{n-\theta p}.
    \end{align*}
    As $\theta p<n$ and since $\eta>0$ is arbitrary, we have that $\vcap_{\theta,p}(\{0\},\,B(0,2r),\,B(0,\Lambda r))=0$, and so $\R^n\setminus\Omega$ does not satisfy the $(\theta,p)$-capacity density condition. 
\end{example}

\end{document}